\documentclass[12pt]{article}
\usepackage{amsthm,amsfonts,amsmath, amscd, bbold}
\usepackage{mathrsfs}
\usepackage{accents}
\usepackage{authblk}

                                \usepackage{verbatim}
\swapnumbers
                              
\theoremstyle{plain}

\usepackage{enumitem}
\numberwithin{equation}{section}
\newtheorem{theorem}{Theorem}[section]

\newtheorem{lemma}[theorem]{Lemma}
\newtheorem{corollary}[theorem]{Corollary}

\theoremstyle{definition}
\newtheorem{definition}[theorem]{Definition}
\newtheorem{example}[theorem]{Example}

\theoremstyle{remark}
\newtheorem{remark}[theorem]{Remark}
\numberwithin{equation}{section}

\newcommand{\R}{{\mathbb R}}

\newcommand{\N}{{\mathbb N}}

\newcommand{\C}{{\mathbb C}}

\newcommand{\bN}{{\mathbb N}}

\newcommand{\cB}{{\mathcal B}}

\newcommand{\cE}{{\mathcal E}}
\newcommand{\cH}{{\mathcal H}}

\newcommand{\cC}{{\mathcal C}}
\newcommand{\cN}{{\mathcal N}}
\newcommand{\cM}{{\mathcal M}}

\newcommand{\cK}{{\mathcal K}}

\newcommand{\one}{{\bf 1}}
\newcommand{\prr}{\sR_\rho}

\newcommand{\Tr}{\mathrm{Tr}}

\newcommand{\be}{\begin{equation}}
\newcommand{\ee}{\end{equation}}
\newcommand{\bea}{\begin{eqnarray}}
\newcommand{\eea}{\end{eqnarray}}
\newcommand{\beann}{\begin{eqnarray*}}
\newcommand{\eeann}{\end{eqnarray*}}

\newcommand{\sP}{{\mathord{\mathscr P}}}

\newcommand{\sE}{{\mathord{\mathscr E}}}

\newcommand{\sA}{{\mathord{\mathscr A}}}
\newcommand{\sR}{{\mathord{\mathscr R}}}

\usepackage{color}

\begin{document}

\title{Recovery map stability for the Data Processing Inequality}

\author[1]{Eric A. Carlen}
\author[2]{Anna Vershynina}
\affil[1]{\small{Department of Mathematics, Hill Center, Rutgers University, 110 Frelinghuysen Road,
Piscataway, NJ 08854-8019, USA}}
\affil[2]{\small{Department of Mathematics, Philip Guthrie Hoffman Hall, University of Houston, 
3551 Cullen Blvd., Houston, TX 77204-3008, USA}}
\renewcommand\Authands{ and }
\renewcommand\Affilfont{\itshape\small}

\date{\today}

\maketitle

\begin{abstract}  The Data Processing Inequality (DPI) says that the 
Umegaki relative entropy $S(\rho||\sigma) := \Tr[\rho(\log \rho - \log \sigma)]$
is non-increasing under the action of completely positive trace preserving (CPTP) maps. 
Let $\cM$ be a finite dimensional von Neumann algebra and $\cN$ a  von Neumann subalgebra of it. Let
$\sE_\tau$ be the tracial conditional expectation from $\cM$ onto $\cN$. For 
density matrices $\rho$ and $\sigma$ in $\cM$, let $\rho_\cN 
:= \sE_\tau \rho$ and $\sigma_\cN := \sE_\tau \sigma$. 
Since $\sE_\tau$ is CPTP, the DPI says that $S(\rho||\sigma) 
\geq S(\rho_\cN||\sigma_\cN)$, and the general case is readily deduced from this. A theorem of Petz says that there is equality if 
and only if $\sigma = \sR_\rho(\sigma_\cN)$, where $\sR_\rho$ is the 
{\em Petz recovery map}, which is dual to the {\em Accardi-Cecchini coarse 
graining operator} $\sA_\rho$ from $\cM$ to $\cN$. We prove a quantitative version of Peta's theorem.  In it simplest form, our bound is
$$S(\rho||\sigma) - S(\rho_\cN||\sigma_\cN)  \geq \left(\frac{\pi}{8}\right)^{4} \|\Delta_{\sigma,\rho}\|^{-2}
\| \sR_\rho(\sigma_\cN) -\sigma\|_1^4\  .\  
$$
where $\Delta_{\sigma,\rho}$ is the relative modular operator. Since  $\|\Delta_{\sigma,\rho}\| \leq \|\rho^{-1}\|$, 
this bound implies a bound with a constant that is independent of $\sigma$. We also  prove an analogous result with a more complicated constant in which the roles of $\rho$ and $\sigma$ are  interchanged on the right.  

Explicitly describing the  solutions set of the Petz equation $\sigma = \sR_\rho(\sigma_\cN)$ 
amounts to 
determining the set of fixed points of the Accardi-Cecchini 
coarse graining map. Building on previous work, we provide a throughly detailed description 
of the set of solutions of the Petz equation $\sR_\rho(\sE_\tau \gamma) = \gamma$, and obtain all of our results in a simple, self contained manner.  Finally, we prove a theorem characterizing state $\rho$ for which the
orthogonal projection from $\cM$ onto $\cN$ in the GNS inner product is a conditional expectation. 
\end{abstract}

\section{Introduction}

\subsection{The Data Processing Inequality}

Let $\cM$ be a finite dimensional von Neumann algebra, which we may regard as a 
subalgebra of  $M_n(\C)$, the  $n\times n$ complex matrices.  The Hilbert-Schmidt inner product $\langle \cdot,\cdot\rangle_{HS}$ on  
$M_n(\C)$ is given  in terms of the trace by
$\langle X,Y\rangle_{HS} = \Tr[X^*Y]$.  Let $\one$ denote the identity.

A state on $\cM$ is a linear functional $\varphi$  on $\cM$ such that $\varphi(A^*A) \geq 0$ for 
$A\in \cM$ and  such that $\varphi(\one) = 1$.  A state $\varphi$  is 
{\em faithful} in case $\varphi(A^*A) > 0$ whenever $A\neq 0$, 
and is  {\em tracial} in case 
$\varphi(AB) = \varphi(BA)$ for all $A,B\in \cM$.  Every state on $\cM$ is of the form 
$X \mapsto \Tr[\rho X]$, where $\rho$ is a {\em density matrix} in $\cM$; i.e., a 
non-negative element $\rho$ of $\cM$ 
such that $\Tr[\rho] =1$. This state is faithful if and only if $\rho$ is invertible.  
It will be convenient to write $\rho(X) = \Tr[\rho X]$ to denote the state corresponding 
to a density matrix $\rho$. Given a faithful state $\rho$, the corresponding 
Gelfand-Naimark-Segal (GNS) inner product is given by $\langle X,Y\rangle_{GNS,\rho} := \rho(X^*Y)$.

In this finite dimensional setting, there is always a faithful tracial state $\tau$  on 
$\cM$, 
namely the one whose density matrix is $n^{-1}\one$.  
The symbol $\tau$ is reserved 
throughout for this tracial state.

Let $\cN$ be a von Neumann subalgebra of $\cM$. Let $\sE$ be any norm-contractive projection 
from $\cM$ onto $\cN$. (Norm contractive means  that $\|\sE (X)\| \leq \|X\|$ for all $X\in \cM$. Throughout the paper, $\|\cdot\|$ 
without any subscript denotes the operator norm.)
By a theorem of Tomiyama \cite{To57}, $\sE$ preserves positivity, $\sE (\one) = \one$,
 and 
\begin{equation}\label{tomi}
 \sE(AXB) = A\sE(X) B \qquad{\rm for\ all} \quad A,B \in \cN, \ X\in \cM\ . 
 \end{equation}
Moreover, as Tomiyama noted, it follows from \eqref{tomi} and the positivity  preserving 
 property of $\sE$ that 
 \begin{equation}\label{tomi2}
 \sE(X)^* \sE(X) \leq \sE(X^*X)  \qquad{\rm for\ all} \ X\in \cM\ ,
 \end{equation}
 In fact, more is true. As is well known, every norm contractive projection is {\em completely positive}. 
 
A {\em conditional expectation} from $\cM$ onto $\cN$, in the sense of Umegaki \cite{U54,U56,U59}, is  a unital projection from $\cM$ onto $\cN$ that is order preserving and such that 
\eqref{tomi} and \eqref{tomi2} are satisfied. 
 Since every conditional expectation $\sE$ is a unital completely positive map,  its adjoint with respect to the Hilbert-Schmidt inner product,  $\sE^\dagger$, is a completely positive trace preserving (CPTP) map, also known as a {\em quantum operation}. (Throughout this paper,  a dagger $\dagger$  always  denotes the adjoint with respect to the Hilbert-Schmidt inner product.)
 
 Let $\sE_\tau$ denote the orthogonal projection from $\cM$ onto $\cN$ with respect to the GNS inner product determined by $\tau$.
 It is easy to see, using the tracial nature of $\tau$, that  $\sE_\tau$ is in fact a conditional expectation, 
and since $\sE_\tau = \sE_\tau^\dagger$, $\sE_\tau$ is a quantum operation.

\begin{definition} For any state $\rho$ on $\cM$, $\rho_\cN$ denotes the state on $\cN$ given by
$\rho_\cN := \sE_\tau(\rho)$
where, as always $\sE_\tau$ denotes the tracial conditional expectation onto $\cN$.
\end{definition}

The restriction of a state $\rho$ on $\cM$ to $\cN$ is of course a state on $\cN$, 
and as such, it is represented by a unique density matrix  belonging to $\cN$, 
which is precisely $\rho_\cN$.

\begin{example}\label{factor} Let $\cH = \cH_1\otimes \cH_2$ be the tensor product of two finite 
dimensional Hilbert spaces. Let $\cM  = \cB(\cH)$ be the algebra of all linear 
transformations on $\cH$, and let 
$\cN$ be the subalgebra $\one_{\cH_1} \otimes \cB(\cH_2)$ consisting of all operators in 
$\cM$ of the form $\one_{\cH_1} \otimes A$, $A\in \cB(\cH_2)$. Then for the normalized trace $\tau$, 
$\cE_\tau(X) = d_1^{-1} \one_{\cH_1}\otimes \Tr_1 X$ for all $X\in\cM$ where $d_1$ is the dimension of $\cH_1$ and where
$\Tr_1$ denotes the partial trace over $\cH_1$.
\end{example}

\if false
\begin{example}\label{nonfactor} Let $\cM$ be the algebra of $2\times 2$ matrices, i.e. algebra generated by the identity and the  Pauli matrices $\cM=\{I, \sigma_1, \sigma_2, \sigma_2\}$. Let $\cN=\{I, \sigma_1\}$ be the subalgebra of $\cM$ generated by the identity and the Pauli matrix $\sigma_1=\left( {\begin{array}{cc}
   0 & 1 \\
   1 & 0 \\
  \end{array} } \right)$. Then, for the normalized trace $\tau$,  the conditional expectation $\cE_\tau$ acts on the operator $X=\alpha_0I+\sum_{j=1}^3\alpha_j\sigma_j$ as follows: $\cE_\tau(X)=\alpha_0I+\alpha_1\sigma_1$. This example cannot be subsumed under Example~\ref{factor} since the algebra $\cN$ is commutative.  More elaborate examples of this type in which neither $\cM$ nor $\cN$ has a trivial center (as $\one_{\cH_1} \otimes \cB(\cH_2)$ always does) arise in the description of systems with finitely many fermion degree of freedom; i.e., in the context of Clifford algebras \cite{CL93}. Clifford algebras over an odd dimensional inner product space have a non-trivial center.
  \end{example}
  \fi

Given two states $\rho$ and $\sigma$ on $\cM$, the {\em Umegaki 
relative entropy  of $\rho$ with respect to $\sigma$} is defined \cite{U62} by
\begin{equation}\label{umegaki}
S(\rho||\sigma) := \Tr[\rho(\log\rho - \log \sigma)]\ . 
\end{equation}
Lindblad's inequality \cite{L74} states that with $\sE_\tau$ being the tracial conditional expectation onto $\cN$, 
\begin{equation}\label{Lindblad1}
S(\rho||\sigma) \geq  S(\sE_\tau(\rho)||\sE_\tau(\sigma)) \ . 
\end{equation}

\if false
\begin{remark} In the context of Example~\ref{factor}, note that   $S(\cE_\tau \rho||\cE_\tau \sigma) = 
S(\Tr_1\rho||\Tr_2\sigma)$; the contribution from  factor $d_1^{-1}\one_{\cH_1}$ cancels in the difference of logarithms. Hence when computing relative entropies in the context of Example~\ref{factor}, one may as well regard $\cE_\tau$ as the familiar partial trace. However, as Uhlmann showed \cite{U73}, the representation in the form $\cE_\tau$ has a number of advantages. These will be useful here. 
\end{remark}
\fi

Lindblad showed that the monotonicity \eqref{Lindblad1} is equivalent to the joint convexity of the relative entropy $(\rho,\sigma) \mapsto
S(\rho||\sigma)$, and this in turn is an immediate consequence of Lieb's Concavity Theorem \cite{L73}. In the case that
$\cM  = \cB(\cH_1\otimes \cH_2)$ and 
$\cN = \{ \one_{\cH_1} \otimes A$, $A\in \cB(\cH_2)\}$, \eqref{Lindblad1} was proved by Lieb and Ruskai \cite{LR73}, who showed it to be equivalent to the Strong Subadditivity (SSA) of the von Neumann entropy; more information on SSA is contained in Section~\ref{Sec:SSA}.

Using the fact that Stinesping's Dilation Theorem \cite{St55} relates general CPTP maps to tracial expectation, 
Lindblad \cite{L75} was  able to prove, using \eqref{Lindblad1} that for any CPTP map $\sP$, 
\begin{equation}\label{DPI}
S(\rho||\sigma) \geq  S(\sP(\rho)||\sP(\sigma)) \ . 
\end{equation}
This is known as the {\em Data Processing Inequality} (DPI). Because of the simple relation between 
\eqref{Lindblad1} and \eqref{DPI} the problem of determining the  cases of equality in the 
Data Processing Inequality largely comes down the problem of determining the cases of equality in 
\eqref{Lindblad1}, which was solved by Petz \cite{P86b,P88}.   His necessary and sufficient  condition for equality in 
\eqref{Lindblad1} is closely connected with the problem of {\em quantum coarse graining}, and in particular a quantum coarse graining operation introduced by Acardi and Cecchini \cite{AC82}, whose dual is now now  known as the {\em Petz recovery channel}, the CPTP map $\sR_\rho$ given by 
\begin{equation}\label{ac14}
\sR_\rho(\gamma) = \rho^{1/2} (\rho_\cN^{-1/2}\gamma\rho_\cN^{-1/2})\rho^{1/2}\ .
\end{equation}
It is obvious that $\sR_\rho(\rho_\cN) = \rho$, so that $\sR_\rho$ ``recovers'' $\rho$ from $\rho_\cN$. Petz proved \cite{P86b,P88} that
there is equality in \eqref{Lindblad1} if and only if  
\begin{equation}\label{pcon1}
\sR_\rho(\sigma_\cN) = \sigma
\end{equation}
and that this is true if and only if 
\begin{equation}\label{pcon2}
\sR_\sigma(\rho_\cN) = \rho\ .
\end{equation}

There has been much recent work on {\em stability} for for the DPI: \cite{ FR15,JRSWW15,S14,SBT17,S18,Z16}. Suppose that $\rho$ and $\sigma$  are such that there is {\em approximate} equality in \eqref{Lindblad1}. To what extent do $\rho$ and $\sigma$ provide approximate solutions to Pets's equation \eqref{pcon1} and \eqref{pcon2}?  

\subsection{Main results}

In this paper we further develop an approach that we introduced in \cite{CV18}  for proving stability for analogs of the DPI for R\'enyi relative entropies.  The R\'enyi relative entropies include the Umegaki relative entropy \eqref{umegaki} as a limiting case, but taking advantage of the special structure on the Umegaki relative entropy, we are able to sharpen the stability bounds obtained in \cite{CV18} for this case.  Our results in this direction are given in Theorem~\ref{thm:main2} and Corollary~\ref{cl:main1}.  These results are  proved in Section 2. 

Moreover, we show how the approach introduced in \cite{CV18} yields a simple specification of the structure of the pairs of densities
$\rho$ and $\sigma$ that satisfy Petz's equations \eqref{pcon1} and \eqref{pcon2}.  This is a question that was first dealt with by Hayden, Jozsa, Petz and Winter in \cite{H03}. The question arose there because, although it is known (and explained in the appendix) that the DPI is 
equivalent to the quantum Strong Subadditivity of the Entropy (SSA) \cite{LR73}, it is not so simple to translate Petz's condition for equality in the DPI into the condition for equality in SSA. This was done in \cite{H03}, but Petz himself, together with Mosnyi, returned to the problem of determining the structure of solutions of \eqref{pcon1} and \eqref{pcon2}  in \cite{MP04}.  Here we give simple specification of the structure of solutions providing some new information.  Our main results in this direction are given in Theorem~\ref{thm:main} and Corollary~\ref{cl:main3}.  These results are proved in Section 3. 

As explained in the next two  subsections, the questions considered here are closely related to  several questions concerning quantum conditional expectations and quantum ``coarse graining''.  in Section 4, we prove a theorem specifying conditions under which the orthogonal projection of $\cM$ onto $\cN$ is a conditional expectations, and hence, when expectation preserving conditional expectations exist. The full statement is given in Theorem~\ref{takvar}.

Before giving precise statements, it will be useful to recall the relatively simple situation regarding the classical DPI.

\subsection{The  classical DPI}

It will be useful to recall some aspects of the classical analog of the DPI. Let $\Omega$ be a finite set. Let $\mathcal{F}$ be a non-trivial partition of $\Omega$. Let $\cM$ denote the functions on $\Omega$, and let $\cN$ the functions on $\Omega$ that are constant on each set of the partition $\mathcal{F}$.  Then $\cM$ and $\cN$ are commutative von Neumann algebras, and $\cN$ is a subalgebra of $\cM$. Let $X$ be a function on $\Omega$ such that $X(\omega) = X(\omega')$ if and only $\omega$ and $\omega'$ belong to the same set in $\mathcal{F}$.  Then $X$ {\em generates} $\mathcal{F}$ in the sense that the sets constituting $\mathcal{F}$ are precisely the non-empty sets of the form $\{\omega\:\ X(\omega) = x\}$. 

Let $\rho$ and $\sigma$ be two strictly positive  probability densities on the set $\Omega$. Let $\tau$ denote the uniform probability density on $\Omega$; i.e., $\tau(\omega) = |\Omega|^{-1}$ for all $\omega$, where $|\Omega|$ is the cardinality of $\Omega$.   As above, let  $\sE_\tau$ denote the orthogonal projection of $\cM$ onto $\cN$, which is nothing other than the conditional expectation with respect  the random variable $Y$ and the probability measure $\tau$. As above, let
$\rho_\cN= \sE_\tau \rho$ and  $\sigma_\cN = \sE_\tau\sigma$.  It is clear the $\rho_\cN$ is a ``coarse grained'' version of $\rho$, obtained by averaging $\rho$ on the sets of the partition $\mathcal{F}$, making it constant on these.

Let $f(\omega|y)$ be the conditional density under $\rho$  for $\omega$ given $Y(\omega)$, and likewise let  $g(\omega|y)$ be the conditional density under $\sigma$  for $\omega$ given $Y(\omega)$. That is
\begin{equation}\label{pins0}
f(\omega | x) = \frac{\rho(\omega)}{\rho_\cN(x)}  \quad{\rm and}\quad  g(\omega | x) = \frac{\sigma(\omega)}{\rho_\cN(x)}\ ,
\end{equation}
which, for each $x$ in the range of $X$ are both probability densities on the set  $\{\omega\:\ X(\omega) = x\}$.  

Then we have $\rho(\omega) = \rho_\cN(X(\omega)) f(\omega|X(\omega))$ and  $\sigma(\omega) = \sigma_\cN(X(\omega)) g(\omega|X(\omega))$, and hence
\begin{eqnarray}\label{chain}
S(\rho ||\sigma) &=& \sum_{\omega\in \Omega} \rho(\omega)(\log \rho(\omega) - \log \sigma(\omega))\nonumber\\
&=&  \sum_{\omega\in \Omega} \rho(\omega)([\log \rho_\cN(\omega) - \log \sigma_\cN(\omega)] +  [\log f(\omega|X(\omega)) - \log g(\omega|X(\omega))  ])\nonumber\\
&=& S(\rho_\cN||\sigma_\cN) +  \sum_{\omega\in \Omega}\rho_n(X(\omega)) \left[  f(\omega |X(\omega))  \log f(\omega |X(\omega)) - \log g(\omega |X(\omega)) \right]
\end{eqnarray}
For each $x$ in the range of $X$, it follows from Jensen's inequality that
\begin{equation}\label{pins}
\sum_{\{\omega\ :\ X(\omega) = x\}}  f(\omega |x)  \left[  \log f(\omega | x) - \log g(\omega |x) \right] \geq 0\ ,
\end{equation}
and there is equality if and only if $f(\omega|x) = g(\omega |x)$ everywhere on $\{\omega\ :\ X(\omega) = x\}$. 
It follows that  $ S(\rho ||\sigma) \geq S(\rho_\cN||\sigma_\cN)$ with equality if and only if for each $x$ in the range of $X$, 
$f(\omega|x) = g(\omega |x)$ everywhere on $\{\omega\ :\ X(\omega) = x\}$.

In this case, $X$ is called a {\em sufficient statistic} for the pair $\{\rho,\sigma\}$. Suppose we are given an independent identically distributed sequence of points  $\{\omega_j\}$, drawn according to one of the two probability densities $\rho$ or $\sigma$, and we want to determine which it is.  If we know the function $f(\omega | x)$, it suffices to observe the sequence $\{X(\omega_j)\}$, and to determine  which of $\rho_\cN(x)$ or $\sigma_\cN(x)$ is governing its distribution. 

Indeed, we can define a classical recovery map as follows: For any probability density  $\gamma \in \cN$, regarded as a probability density on the range of $X$, define $\sR_\rho \gamma$ to be the probability density in $\cM$ given by
$$\sR_\rho \gamma(\omega) = \gamma(X(\omega))f(\omega | X(\omega))\ .$$
Therefore, we can express the condition for equality in the classical DPI as $\sR_\rho \sigma_\cN = \sigma$, and evidently this is true if and only if $\sR_\sigma \rho_\cN = \rho$. This is the classical analog of Petz's result. Moreover, in this notation, we have that 
$$
\sum_{\omega\in \Omega}\rho_n(X(\omega)) \left[  f(\omega |X(\omega))  \log f(\omega |X(\omega)) - \log g(\omega |X(\omega)) \right]  = S(\rho ||\sR_\sigma \rho)\ ,
$$
so that \eqref{chain} becomes
\begin{equation}\label{chain2}
S(\rho||\sigma) - S(\rho_\cN||\sigma_\cN) \geq   S(\rho ||\sR_\sigma \rho)\ .
\end{equation}

Then by the  classical Pinsker inequality, 
\begin{equation}\label{chain3}
S(\rho||\sigma) - S(\rho_\cN||\sigma_\cN) \geq   \frac12 \left(\sum_{\omega \in \Omega} |\rho(\omega) - \sR_\sigma \rho(\omega)|\right)^2\ .
\end{equation}

It remains an open problem to prove quantum  analogs of \eqref{chain2} or \eqref{chain3}, even with worse constants on the right. Here we prove a quantum  analog of \eqref{chain3} with a worse constant, and with the power raised from 2 to 4 on the right. The line of  argument has to be entirely different from the one we have just employed in the classical case because there is no effective quantum replacement for ``conditioning on the observable $X$''.  

It is therefore useful to find a way of describing the classical recovery map that doe not refer explicitly to conditioning on the random variable $X$.   Define $\sE_\sigma$  to be the orthogonal projection of $\cM$ onto $\cN$ in $L^2(\sigma)$. Then for any random variable $Y$ (i.e., any function on $\Omega$), $\sE_\sigma Y$ is the condition expectation of $Y$ given the sigma-algebra $\cM$.   The operation $Y \mapsto \sE_\sigma Y$ 
yields a ``coarse grained version'' of $Y$ that is constant on the sets in $\mathcal{F}$: With $X$ and $g$ as in \eqref{pins0},
$$\sE_\sigma Y(\omega) =   \sum_{\omega' :\ X(\omega') = X(\omega)\}} g(\omega' | X(\omega)Y(\omega')\ .$$
It is clear from this formula that $\sE_\sigma$ preserves positivity, and preserves expectations with respect to $\sigma$, That is,
\begin{equation}\label{pins6}
\sigma(Y) = \sigma (\sE_\sigma Y)\ .\end{equation}
Now let $\sE_\sigma^\dagger$ be the dual operation taking states on $\cM$ (probability densities on the range of $X$) to states on $\cM$ (probability densities on $\Omega$).  
It is easily seen that this is nothing other than $\sR_\sigma$.  That is, the classical recovery map $\sR_\sigma$  is nothing other than the dual of the conditional expectation $\sE_\sigma$, which is nothing other than the orthogonal projection of $\cM$ onto $\cN$ in $L^2(\sigma)$. This analytic specification of $\sR_\sigma$, making no explicit mention of conditioning on $X$, provides a starting point for the construction of a quantum recovery map. 

\subsection{Quantum conditional expectations and quantum coarse graining}

The discussion of the classical DPI brings us to the question as to whether for any faithful state $\rho$ on $\cM$ there exists
a conditional expectation $\sE$  from $\cM$ onto $\cN$ that preserves expectations with respect to $\rho$, i.e.
such that
\begin{equation}\label{tomi3}
\rho (X)  = \rho (\sE(X))  \qquad{\rm for\ all} \ X\in \cM\ .
\end{equation}
The property \eqref{tomi3} says that ``the expectation of a conditional expectation of an 
observable equals the expectation of the observable''.   If such a conditional expectation exists, 
then it is unique:  {\em Any such conditional expectation must be the orthogonal projection of 
$\cM$ on $\cN$ with respect to the $GNS$ inner product for the state $\rho$}. To see this, note that
for all $X\in \cM$ and all $A\in \cN$, using \eqref{tomi}, 
\begin{equation}\label{tomi4}
\langle A,X\rangle_{GNS,\rho} = \rho(A^*X) = \rho(\sE(A^*X)) = \rho( A^*\sE(X)) = 
\langle A, \sE(X)\rangle_{GNS,\rho}\ .
\end{equation}

Suppose that $\sE$ is a conditional expectation satisfying \eqref{tomi3}. Then since $\sE$ is a unital completely positive map, 
$\sE^\dagger$ is a CPTP map; i.e., a quantum channel. For any state $\gamma$ on $\cN$, and any $A\in \cM$, we then have
$$\sE^\dagger(A) = \gamma(\sE(A))\ .$$
Then when \eqref{tomi3} is satisfied, taking $\gamma = \rho_\cN$, we have 
$$\sE^\dagger\rho_\cN(A) = \rho_\cN(\sE(A)) = \rho(A)\ ,$$
and this means that $\sE^\dagger$ is a quantum channel that ``recovers'' $\rho$ from $\rho_\cN$.

As we have already noted, $\sE_\tau$ is a conditional expectation with the property \eqref{tomi3}. 
However, for non-tracial states $\rho$, a conditional expectation satisfying \eqref{tomi3} need not exist. 

A theorem of 
Takesaki \cite{Ta72} says, in our finite dimensional context,  that for a faithful state $\rho$, there exists a conditional  
expectation $\sE$ from $\cM$ onto $\cN$ if and only if  
$\rho A\rho^{-1}\in \cN$ for all $A \in \cN$, and in general this is not the case. We give a short proof of this 
and somewhat  more in Section~\ref{Sec:Exp}: In Theorem~\ref{takvar}, we prove that $\sE_\rho$, 
the orthogonal projection from $\cM$ onto 
$\cN$ in the GNS inner product with respect to $\rho$, is {\em real}
(that is, it preserves self-adjointness) if and only if  $\rho A\rho^{-1}\in \cN$ for all $A \in \cN$.
Since every order preserving linear transformation is real, this precludes the
general existence of conditional expectations satisfying \eqref{tomi3} whenever $\cN$ is not invariant under $X\mapsto \rho X \rho^{-1}$, thus 
implying Takesaki's Theorem (in this finite dimensional setting).

There is another inner product on $\cM$ that is naturally induced by a  faithful state $\rho$, 
namely the {\em Kubo-Martin-Schwinger (KMS) inner product}. It is defined by
\begin{equation}\label{KMSip}
\langle X,Y\rangle_{KMS,\rho} = \Tr[\rho^{1/2}X^*\rho^{1/2}Y] = \Tr[(\rho^{1/4}X\rho^{1/4})^*
(\rho^{1/4}Y\rho^{1/4})]\ .
\end{equation}

Evidently, for any $X\in \cM$ and $Y\in \cN$, 
\begin{eqnarray*}
| \langle X,Y\rangle_{KMS,\rho}| = 
|\Tr[\rho^{1/2}X^*\rho^{1/2}Y] | &=& 
|\Tr[(\rho_\cN^{-1/4} \rho^{1/2}X\rho^{1/2}\rho_\cN^{-1/4})^*(\rho_\cN^{1/4}Y\rho_\cN^{1/4})] |\\ 
&\leq& \| (\rho_\cN^{-1/4} \rho^{1/2}X\rho^{1/2}\rho_\cN^{-1/4})\|_{HS} \| Y\|_{KMS,\rho_\cN}\ .
\end{eqnarray*}
Hence $Y\mapsto \langle X,Y\rangle_{KMS,\rho}$ is a bounded linear functional on 
$(\cN, \langle\cdot\, ,\cdot\rangle_{KMS,\rho_\cN})$, and then  there is a uniquely determined $\sA_\rho(X)\in \cN$ such 
that for all $X\in \cM$ and all $Y\in \cN$,
\begin{equation}\label{ac0}
 \langle X,Y\rangle_{KMS,\rho} = \langle \sA_\rho(X),Y\rangle_{KMS,\rho_\cN}\ .
\end{equation}

The map $\sA_\rho$ was  introduced by Accardi and Cecchini \cite{AC82}, building on previous work by Accardi \cite{A74}.

\begin{definition} Let $\rho$ be a faithful state on $\cM$. The 
{\em Accardi-Cecchini coarse graining operator} 
$\sA_\rho$  from $\cM$ to $\cN$ is defined by  \eqref{ac0}. 
\end{definition}

The map $\sA_\rho$ was  introduced by Accardi and Cecchini \cite{AC82}, building on previous work by Accardi \cite{A74}.  It is a ``coarse graining'' operation in that to each observable $X$  in the larger algebra $\cM$, it associates an observable $\sA_\rho(X)$, in the smaller algebra 
$\cN$, and measurement of $\sA_\rho(X)$ will yields coarser information than a measurement of $X$ itself. The same, of course, is true for conditional expectations, 

Since $\one \in \cN$ by definition,
for all $X\in \cM$, 
$\langle  \one,X \rangle_{KMS,\rho} = \langle  \one,\sA_\rho(X) \rangle_{KMS,\rho}$, and for all 
$X\in \cM$, $\langle \one, X\rangle_{KMS,\rho} = 
\Tr[\sigma^{1/2}\one \sigma^{1/2}X] = \rho(X)$. Therefore
\begin{equation}\label{ac1}
\rho(\sA_\rho(X)) = \rho(X)\ .
\end{equation}
Thus, unlike conditional expectations in general, the Accardi-Cecchini coarse-graining operator {\em always} preserves expectations 
with respect to $\rho$.  

In the matricial setting, it is a particularly simple matter to derive an explicit expression for $\sA_\rho$.
By definition,  for all $X\in \cM$ and all $Y\in \cN$,
\begin{equation}\label{ac10}
\Tr[\rho_\cN^{1/2} Y \rho_\cN^{1/2} \sA_\rho(X)] = \Tr[\rho^{1/2}Y\rho^{1/2} X]\ .
\end{equation}
Make the change of variables $Z = \rho_\cN^{1/2} Y \rho_\cN^{1/2}$. 
Since $\rho_\cN^{1/2}$ is invertible and $Y$ ranges over $\cN$, $Z$ ranges over $\cN$. Hence
\begin{equation}\label{ac11}
\Tr[Z \sA_\rho(X)] = 
\Tr[(\rho^{1/2}\rho_\cN^{-1/2} Z \rho_\cN^{-1/2}\rho^{1/2}) X] = 
\Tr[Z(\rho_\cN^{-1/2}\rho^{1/2} X \rho^{1/2}\rho_\cN^{-1/2})]\ .
\end{equation}
Since the above holds for all $Z\in\cN$, it follows that
\begin{equation}\label{ac12}
\sA_\rho(X) = \rho_\cN^{-1/2}\sE_\tau(\rho^{1/2} X \rho^{1/2})\rho_\cN^{-1/2}\ .
\end{equation}
It is evident from this formula that $\sA_\rho$ is a completely positive 
unital map from $\cM$ to $\cN$, and therefore it is actually a contraction from $\cM$ to $\cN$.
 By Tomiyama's Theorem, it cannot in general be a projection of $\cM$ onto $\cN$. That is, if 
 $X\in \cN$, it is not necessarily the case that $\sA_\rho(X) = X$. The set of 
 $X\in \cN$ for which this is true turns out to be a subalgebra of $\cN$, as was shown by 
 Accardi and Cecchini \cite{AC82}. This subalgebra will be of interest in what follows. 
 
 \begin{definition} The {\em Petz recovery map} $\sR_\rho$ is the
 Hilbert-Schmidt adjoint of $\sA_\rho$ \cite{P88}. That is, $\sR_\rho = \sA_\rho^\dagger$, or equivalently, for all density matrices $\gamma\in \cN$
 $$
\Tr[\gamma\sA_\rho(X)] =  \Tr[\sR_\rho(\gamma)X]\ . 
 $$
 {\em A dagger $\dagger$  always  denotes the adjoint with respect to the Hilbert-Schmidt inner product}.
 \end{definition}
 As the dual of a unital completely positive map, $\sR_\rho$ is a CPTP map. Moreover, 
 it follows immediately from the definition and \eqref{ac12}  that for all density matrices $\gamma\in \cN$, 
 \begin{equation}\label{ac14}
\sR_\rho(\gamma) = \rho^{1/2} (\rho_\cN^{-1/2}\gamma\rho_\cN^{-1/2})\rho^{1/2}\ .
\end{equation}
 It is evident from this formula not only that $\sR_\rho$ is a CPTP map, but that $\sR_\rho(\rho_\cN) = \rho$; i.e., 
 $\sR_\rho$ recovers $\rho$ from $\rho_\cN$. Now suppose that $\sigma$ is another density matrix in $\cM$ and that
\begin{equation}\label{Petzeq}
\sR_\rho(\sigma_\cN) = \sigma\ .
\end{equation} 
Then by the Data Processing Inequality and \eqref{Petzeq},
$S(\rho||\sigma) \leq S(\sR_\rho(\rho_\cN) ||\sR_\rho(\sigma_\cN) ) = S(\rho||\sigma)$.
Hence when $\sR_\rho(\sigma_\cN) = \sigma$, there is equality in \eqref{Lindblad1}. The deeper result of Petz \cite{P86b,P88} is that there is equality in \eqref{Lindblad1} {\em only} in this case. 

Our goal  is to prove a stability bound for Petz's theorem on the cases of
equality in \eqref{Lindblad1}. Our result involves the {\em relative modular operator}
$\Delta_{\rho,\sigma}$ on $\cM$ defined by
\begin{equation}\label{relmod}
\Delta_{\sigma,\rho}(X) = \sigma X \rho^{-1}
\end{equation}
for all $X\in \cM$. This is the matricial version of an operator introduced in a more general von Neumann algebra context 
by Araki \cite{A76}.  Our main result is:

\begin{theorem}\label{thm:main2}
Let  $\rho$ and  $\sigma$ be two states on $\cM$ Let 
$\sE_\tau$ be the tracial conditional expectation onto a von Neumann subalgebra $\cN$, and let 
$\rho_\cN = \sE_\tau \rho$ and $\sigma_\cN = \sE_\tau \sigma$. Then, with $\|\cdot\|_2$ denoting the Hilbert-Schmidt norm, 
\begin{equation}\label{REM5a}
S(\rho||\sigma) - S(\rho_\cN||\sigma_\cN)  \geq \left(\frac{\pi}{4}\right)^{4} \|\Delta_{\sigma,\rho}\|^{-2}
\|\sigma_\cN^{1/2}\rho_\cN^{-1/2} \rho^{1/2} - \sigma^{1/2}\|_2^4\  .
\end{equation}
\end{theorem}

The quantity on the right hand side may be estimated in terms of the Petz recovery map. In Section~\ref{Sec:Stab} we prove:

\begin{lemma}\label{HStoTr} Let $\rho$, $\sigma$ and $\sigma_\cN$ be specified as in Theorem~\ref{thm:main2}. Then, 
with $\|\cdot\|_1$ denoting the trace norm,
$$\|(\sigma_\cN)^{1/2}(\rho_\cN)^{-1/2} \rho^{1/2} - 
\sigma^{1/2}\|_2 \geq \frac12 \| \prr(\sigma_\cN) -\sigma\|_1\ .$$
\end{lemma}

As an immediate Corollary of Theorem~\ref{thm:main2} and Corollary~\ref{HStoTr}, we obtain
\begin{corollary}\label{cl:main1} Let  $\rho$ and  $\sigma$ be two states on $\cM$. Let 
$\sE_\tau$ be the tracial conditional expectation onto a von Neumann subalgebra $\cN$, and let 
$\rho_\cN = \sE_\tau \rho$ and $\sigma_\cN = \sE_\tau \sigma$. 
Then, with $\|\cdot\|_1$ denoting the trace norm,
\begin{equation}\label{REM5b}
S(\rho||\sigma) - S(\rho_\cN||\sigma_\cN)  \geq \left(\frac{\pi}{8}\right)^{4} \|\Delta_{\sigma,\rho}\|^{-2}
\| \sR_\rho(\sigma_\cN) -\sigma\|_1^4\  .
\end{equation}
\end{corollary}

Note how the right hand side of \eqref{REM5b} differs from the right hand side of \eqref{chain3}: Apart from the constant and the power $4$ in place of $2$, the most striking difference is that the roles of $\rho$ and $\sigma$ are reversed. The expected result, with a worse constant, is obtained in Corollary~\ref{cl:symm}. 

Recall that the modular operator is the right multiplication by $\rho^{-1}$ and left multiplication by $\sigma$, so $\|\Delta_{\sigma,\rho}\|\leq \|\rho^{-1}\|$, since $\|\sigma\|\leq 1.$ While $\|\rho^{-1}\|$ might be considerably larger than 
$\|\Delta_{\sigma,\rho}\|$, a bound in terms of $\|\rho^{-1}\|$ has the merit that it is
 independent of $\sigma$:
\begin{equation}\label{REM5c}
S(\rho||\sigma) - S(\rho_\cN||\sigma_\cN)  \geq \left(\frac{\pi}{8}\right)^{4} \|\rho^{-1}\|^{-2}
\| \sR_\rho(\sigma_\cN) -\sigma\|_1^4\  .
\end{equation}

Corollary~\ref{cl:main1}  yields a result of Petz: With $\cM$, $\cN$, $\rho$ and $\sigma$ as above, 
$S(\rho||\sigma) = S(\rho_\cN||\sigma_\cN)$ if and only if $\sigma$ satisfies the Petz equation
\begin{equation}\label{equal}
    \sigma =  \sR_\rho(\sigma_\cN)\ .
\end{equation}

Theorem~\ref{thm:main2} gives what appears to be a stronger condition on relating $\rho$, $\sigma$, $\rho_\cN$ and $\sigma_\cN$, namely that 
\begin{equation}\label{eqcaseS}
\rho_\cN^{-1/2} \rho^{1/2} = \sigma_\cN^{-1/2}\sigma^{1/2}\ .
\end{equation}
While validity of \eqref{eqcaseS} immediately implies that $\sigma$ satisfies the Petz equation \eqref{equal}, the converse is also true:  By what we have noted above, when \eqref{equal} is satisfied,
$S(\rho||\sigma) = S(\rho_\cN||\sigma_\cN)$, and then by  Theorem~\ref{thm:main2}, \eqref{eqcaseS}
is satisfied.  

This may be made quantitative as follows:  Letting $L_A$ denote the operator of left multiplication by $A$,
$$L_{\rho_\cN^{1/2}} L_{\sigma_\cN^{-1/2}}(\sigma_\cN^{1/2}\rho_\cN^{-1/2} \rho^{1/2} - \sigma^{1/2}) = 
( \rho^{1/2} -  \rho_\cN^{1/2}  \sigma_\cN^{-1/2}\sigma^{1/2}) \ ,$$
and hence
\begin{equation}\label{compA}
\| \rho^{1/2} -  \rho_\cN^{1/2}  \sigma_\cN^{-1/2}\sigma^{1/2}\|_2 \leq \|L_{\rho_\cN^{1/2}}\| \| L_{\sigma_\cN^{-1/2}}\| \|\sigma_\cN^{1/2}\rho_\cN^{-1/2} \rho^{1/2} - \sigma^{1/2}\|_2\ .
\end{equation}
Since $\|L_{\rho_\cN^{1/2}}\| = \|\rho_{\cN}\|^{1/2}$ and $\| L_{\sigma_\cN^{-1/2}}\|  = \|\sigma_\cN^{-1}\|^{1/2}$,
we may combine \eqref{compA} with \eqref{REM5a} to obtain

\begin{equation}\label{REM5a2}
S(\rho||\sigma) - S(\rho_\cN||\sigma_\cN)  \geq \left(\frac{\pi}{4}\right)^{4} \|\Delta_{\sigma,\rho}\|^{-2}
\|\rho_{\cN}\|^{-2}\|\sigma_\cN^{-1}\|^{-2}
\|\rho_\cN^{1/2}\sigma_\cN^{-1/2} \sigma^{1/2} - \rho^{1/2}\|_2^4\  ,
\end{equation}
which is the analog of \eqref{REM5a} with a somewhat worse constant on the right, but the roles of $\rho$ and $\sigma$ interchanged there.  Applying Lemma~\ref{HStoTr} once more, we obtain

\begin{corollary}\label{cl:symm} Let  $\rho$ and  $\sigma$ be two states on $\cM$. Let 
$\sE_\tau$ be the tracial conditional expectation onto a von Neumann subalgebra $\cN$, and let 
$\rho_\cN = \sE_\tau \rho$ and $\sigma_\cN = \sE_\tau \sigma$.  Then
\begin{equation}\label{REM5b2}
S(\rho||\sigma) - S(\rho_\cN||\sigma_\cN)  \geq \left(\frac{\pi}{8}\right)^{4} \|\Delta_{\sigma,\rho}\|^{-2}
\|\rho_{\cN}\|^{-2}\|\sigma_\cN^{-1}\|^{-2}
\| \sR_\sigma(\rho_\cN) -\rho\|_1^4\  .
\end{equation}
\end{corollary}
As above, bounding the norms of states by 1, we get a constant that depends only on the smallest eigenvalues of $\rho$ and $\sigma_\cN$
\begin{equation}\label{REM5c2}
S(\rho||\sigma) - S(\rho_\cN||\sigma_\cN)  \geq \left(\frac{\pi}{8}\right)^{4} \|\rho^{-1}\|^{-2}
\|\sigma_\cN^{-1}\|^{-2}
\| \sR_\sigma(\rho_\cN) -\rho\|_1^4\  .
\end{equation}

We noted above that   $\sigma$ solves the Petz equation if and only if 
\eqref{eqcaseS} is satisfied, and then  since \eqref{eqcaseS} is symmetric in $\rho$ and $\sigma$, 
$\sigma = \sR_\rho \sigma_\cN$ if and only if  $\rho = \sR_\sigma \rho_\cN$, and hence
\begin{equation}\label{symmb}
S(\rho||\sigma) = S(\rho_\cN||\sigma_\cN)  \iff  S(\sigma||\rho) = S(\sigma_\cN||\rho_\cN)\ .
\end{equation}
The reasoning leading to Corollary~\ref{cl:symm} show that moreover, $\|\rho = \sR_\sigma \rho_\cN\|_1$ and 
$\|\sigma = \sR_\rho \sigma_\cN\|_1$ are comparable in size.

To state our results on the structure of the solution set of the Petz equation, we introduce the fixed point  set
\begin{equation}\label{ccdef}
\cC := \{ X\in \cM :\ \sA_\rho(X) = X\ \}. 
\end{equation}
Standard results (see Section~\ref{Sec:Struc}) show that $\cC$ is a von 
Neumann subalgebra of $\cN$. Let $\Delta_\rho$ denote the {\em modular operator} on $\cM$,
\begin{equation}\label{modularde}
\Delta_\rho(X) = \rho X \rho^{-1}\ .
\end{equation}
Then $\cC$ may also be characterized (Theorem~\ref{largest}) as the largest von Neumann 
subalgebra of $\cN$ that is invariant under $\Delta_\rho$. 

Let ${\mathcal Z}$ denote the center of $\cC$. Then by standard results 
(see Section~\ref{Sec:Struc}), ${\mathcal Z}$ is generated by a finite family 
$\{P_1,\dots, P_J\}$ of mutually orthogonal projections.  Define $\cH^{(j)}$, 
$j= 1,\dots,J$, to be the range of $P_j$. The restriction of $\cC$ to 
each $\cH^{(j)}$ is a {\em factor}, and hence
each $\cH^{(j)}$ factors as 
$\cH_{j,\ell}\otimes \cH_{j,r}$, and the general element $A$  of $\cC$ has the form
$$A = \sum_{j=1}^J \one_{\cH_{j,\ell}}\otimes A_{j,r}\ ,$$
where $A_{j,r}\in\cB(\cH_{j,r})$. All of this follows from the standard 
theory of the structure of finite dimensional von Neumann algebras, 
and we emphasize that $\cC$ and the decomposition
$\cH = \bigoplus_{j=1}^J \cH_{j,\ell}\otimes \cH_{j,r}$
is canonically associated to $\rho$.   

Since $\cC$ is invariant under $\Delta_\rho$, the orthogonal projection from $\cM$ onto $\cN$ in the
GNS inner product induced by $\rho$ is a conditional expectation (see Theorem~\ref{takvar})
that we denote by 
$\cE_{\cC,\rho}$, and  which we  call  the conditional expectation given $\cC$ under 
$\rho$.
We shall prove:

\begin{theorem}\label{thm:main} Let   $\rho$ be a faithful state on $\cM$, 
and let $\cN$ be a von Neumann subalgebra of $\cM$.   Let $\cC$ be the fixed-point 
algebra of the Accardi-Cecchini coarse graining operator $\sA_\rho$. Let $\cH$ 
be the finite dimensional Hilbert space on which $\cM$ acts, and let
${\displaystyle 
\cH = \bigoplus_{j=1}^J \cH_{j,\ell}\otimes \cH_{j,r}
}$
induced by the decomposition of $\cC$ as a direct sum of factors.  Then there are
uniquely determined density matrices $\{\gamma_{1,r},\dots,\gamma_{J,r}\}$, 
where $\gamma_{j,r}$ acts on $\cH_{j,r}$  and $\gamma_{j,\ell}\otimes \one_{\cH_{j,r}}\in \cM$ so that
\begin{equation}\label{bonnT1}
\rho = \bigoplus_{j=1}^J \gamma_{j,\ell}\otimes \Tr_{\cH_{j,\ell}}(P_j\rho P_j)\ .
\end{equation}
Moreover, $\sE_\tau (\gamma_{j,\ell}\otimes \one_{\cH_{j,r}})$ has the form
$\widetilde{\gamma}_{j,\ell}\otimes \one_{\cH_{j,r}}$ and 

\begin{equation}\label{bonnT1}
\rho_\cN  = \bigoplus_{j=1}^J \widetilde\gamma_{j,\ell}\otimes \Tr_{\cH_{j,\ell}}(P_j\rho P_j)\ .
\end{equation}

A  state $\sigma$ on $\cM$  solves the Petz equation $\sR_\rho\sigma_\cN = \sigma$ if and only if
for all $X$ in $\cM$, $\sigma(X) = \sigma(\sE_{\cC,\rho})$; i.e., if and only if 
expectations with respect to $\sigma$ are preserved under the the conditional 
expectation given $\cC$ under $\rho$. Every such
state $\sigma$ has the form
${\displaystyle 
\sigma = \bigoplus_{j=1}^J \gamma_{j,\ell}\otimes \Tr_{\cH_{j,\ell}}(P_j\sigma P_j)}$
for the same   $\{\gamma_{1,r},\dots,\gamma_{J,r}\}$
\end{theorem}

\begin{corollary}\label{cl:main3} Let  $\rho, \sigma$ be  faithful states on $\cM$, and let $\cN$ be a 
von Neumann subalgebra of $\cM$.   Let $\cC$ be the fixed-point algebra of 
the Accardi-Cecchini coarse graining operator $\sA_\rho$ for $\cN$.  
Let $\sE_{\cC\rho}$  be the conditional expectation given $\cC$ under $\rho$.
Define $\rho_\cC = \sE_{\cC\rho}\rho$, $\sigma_\cC = \sE_{\cC\rho}\sigma$. Then 
\begin{equation}\label{further}
S(\rho||\sigma) = S(\rho_\cN||\sigma_\cN) \quad\iff\quad  S(\rho||\sigma) = S(\rho_\cC||\sigma_\cC)\ .
\end{equation}
In particular, if $\cC$ is spanned by $\one$, $S(\rho||\sigma) = S(\rho_\cN||\sigma_\cN)$ if and only if
$\rho =\sigma$.
\end{corollary}

We close the introduction with some further comments on recovery map stability 
bounds for the Data Processing Inequality.  In physical applications, instead of the 
trace distance, one often consider an alternative measure of the closeness between 
two quantum states, the fidelity \cite{U76}. For two states $\rho$ and $\sigma$ on 
$\cB(\cH)$, the {\em fidelity} between them is defined as
\begin{equation}\label{eq:fidelity}
F(\rho, \sigma)=\|\sqrt{\rho}\sqrt{\sigma}\|_1^2.
\end{equation}
For all states $\rho$ and $\sigma$, we have $0\leq F(\rho, \sigma)\leq 1$. The fidelity equal to one if and only if the states are equal, and it is equal to zero if and only is the support of $\rho$ is orthogonal to the support of $\sigma$. So in other words, the fidelity is zero when states are perfectly distinguishable, and zero when they cannot be distinguished. Note that the fidelity itself satisfies the monotonicity relation under a completely positive trace preserving maps, but we will not discuss it here. Moreover, there is a relation between the trace distance $\|\rho-\sigma\|_1$ and fidelity
\begin{equation}
1-\sqrt{F(\rho, \sigma)}\leq \frac{1}{2}\|\rho-\sigma\|_1\leq \sqrt{1-F(\rho, \sigma)}\ .
\end{equation}
From here and the Corollary \ref{cl:main1} we obtain the quantitative version of the Petz's Theorem involving the fidelity between states
\begin{equation}\label{eq:fidelity}
S(\rho||\sigma) - S(\rho_\cN||\sigma_\cN)  \geq \left(\frac{\pi}{4}\right)^{4} \|\Delta_{\sigma,\rho}\|^{-2}
\left(1-\sqrt{F(\sigma,  \prr(\sigma_\cN))}\right)^4\  .
\end{equation}
Recent results \cite{FR15, JRSWW15, W15, Z16} provide sharpening of the monotonicity inequality, but the lower bounds provided there involve quantities that are hard to compute, e.g. rotated and twirled Petz recovery maps. for another fidelity type bound not explicitly involving the recover map, see \cite[Theorem 2.2]{CL14}. The appeal of the above bound is that it involves simple distance measure between the original state $\sigma$ and Petz recovered state $ \prr(\sigma_\cN)$. 

 The proof of the theorem~\ref{thm:main2} also implies that satisfaction of the
 Petz equation $\sR_\rho \sigma_\cN = \sigma$ is the necessary and sufficient condition for cases of equality in the monotonicity inequality for a large class of {\em quasi-relative entropies}, as we now briefly explain.
 
 Let $f:(0,+\infty)\rightarrow \mathbb{R}$ be an operator convex function, so that for all $n\in \bN$, and all positive $n\times n$ matrices $A$ and $B$, $f(\tfrac12 A + \tfrac12 B) \leq \tfrac12 f(A)) + \tfrac12 f(B)$. We say that $f$ is strictly operator convex in case there is  equality if and only if $A = B$.   

Petz  \cite{P85}, \cite{P86}  
has defined the \textit{$f$-relative quasi-entropy}  as
\begin{equation}\label{eq:quai-r-e-def} 
S_{f}(\rho||\sigma)=\Tr [ f(\Delta_{\sigma,\rho})\rho]  = \langle \rho^{1/2}, 
f(\Delta_{\sigma,\rho})\rho^{1/2}\rangle_{HS}\ .
\end{equation}
Since 
$-\log(\Delta_{\sigma,\rho})\rho^{1/2} =  \rho^{1/2} \log\rho  - \log\sigma \rho^{1/2}$, 
the choice $f(x) = -\log x$ yields the Umegaki relative entropy. 

Since for each $t>0$, the function $x\mapsto (t+x)^{-1}$ is operator convex, 
this construction  yields a one parameter family a quasi relative entropies, $S_{(t)}$, defined by
\begin{equation}\label{sfrel}
S_{(t)}(\rho||\sigma) =  \Tr\left[ (t + \Delta_{\sigma,\rho})^{-1}\rho\right] \ .
\end{equation}
From the  integral representation of the logarithm
$$-\log (A) = \int_{0}^\infty \left( \frac{1}{t +A }  -\frac{1}{1+t} \right){\rm d}t\ ,$$
it follows that 
\begin{equation}\label{sfrelB}
S(\rho||\sigma) =  \int_0^\infty \left( S_{(t)}(\rho||\sigma) - \frac{1}{1+t}\right){\rm d} t\ ,
\end{equation}
and we may use this representation to study monotonicity for the Umegaki 
relative entropy in terms of monotonicity for the one parameter family of quasi 
relative entropies $S_{(t)}$. The proof of Theorem~\ref{thm:main2} is ultimately 
derived from a stability bound for the variant of the Data Processing Inequality 
that is valid for the quasi relative entropies $S_{(t)}$.  Since the R\'enyi 
relative entropies can be expressed in terms of  a similar integral representation,  the Petz equation
$\sigma = \sR_\rho \sigma_\cN$ again characterizes the condition for 
cases of equality in these variants of the Data Processing Inequality; this will be developed in detail in a companion paper.


\section{Stability for the Data Processing Inequality}\label{Sec:Stab}

We begin this section by recalling Petz's proof of the monotonicity of the 
quasi relative entropies $S_f$ for operator convex $f$. 

Throughout this section $\cN$ is a von Neumann 
subalgebra of the finite dimensional von Neumann algebra $\cM$, and $\rho$ and 
$\sigma$ are two density matrices in 
$\cM$. $\sE_\tau$ is the tracial conditional expectation onto $\cN$, and 
$\rho_\cN = \sE_\tau \rho$ and 
$\sigma_\cN = \sE_\tau \sigma$. Finally $\cH$ denotes $(\cM, \langle \cdot,\cdot\rangle_{HS})$,
 
Define the operator $U$ mapping $\cH$ to $\cH$ by
\begin{equation}\label{Ujdef}
U(X) = \sE_\tau(X)\rho_\cN^{-1/2}\rho^{1/2}\ .
\end{equation}
Note that for all $X\in\cN$, $U(X)=X\rho_\cN^{-1/2}\rho^{1/2}.$
The adjoint operator on $\cH$  is given by
\begin{equation}\label{Ujcal}
U^*(Y) = \sE_\tau(Y\rho^{1/2}) \rho_\cN^{-1/2}\ 
\end{equation}
for all $Y\in \cH = \cM$. 

For $X\in \cM$,   $U^*U(X) = \sE_\tau  (\rho_\cN^{-1/2}\sE_\tau (X)\rho_\cN^{-1/2}\rho)= \sE_\tau(X)$.
Hence 
$U^*U  = \sE_\tau$,  the orthogonal projection in $\cH$ onto $\cN$.  
That is, $U$, restricted to $\cN$,  is an isometric embedding of $\cN$ into $\cH= \cM$, but it is not the trivial isometric embedding by inclusion. Also, we see that on $\cN$ the map $U$ is isometric.

Now observe that for all $X\in \cN$, 
$\Delta_{\sigma,\rho}^{1/2}(U(X))  = \sigma^{1/2} X\rho_\cN^{-1/2}$,
and hence for all $X\in \cN$, 
\begin{eqnarray*}
\langle \Delta_{\sigma,\rho}^{1/2}(U(X)), \Delta_{\sigma,\rho}^{1/2}(U(X))\rangle &=& 
\Tr( (\rho_\cN)^{-1/2} X^* \sigma X (\rho_\cN)^{-1/2})\\
&{=}&\Tr( (\rho_\cN)^{-1/2} X^*\sigma_\cN X (\rho_\cN)^{-1/2})\\
&=&
\langle \Delta^{1/2}_{\sigma_\cN,\rho_\cN}(X), \Delta^{1/2}_{\sigma_\cN,\rho_\cN}(X)\rangle\ .
\end{eqnarray*}
That is, on $\cN$, 
\begin{equation}\label{key}
  U^* \Delta_{\sigma,\rho} U=\Delta_{\sigma_\cN,\rho_\cN} \ .
\end{equation}

By the operator Jensen inequality, as operators on $(\cN, \langle \cdot,\cdot\rangle_{HS})$,\begin{equation}\label{opjen}
  U^* f(\Delta_{\sigma,\rho}) U \geq f\left(   U^* \Delta_{\sigma,\rho} U\right)\ .
\end{equation}
Combining \eqref{key} and  \eqref{opjen}, and using  the fact that $U (\rho_\cN)^{1/2} = \rho^{1/2}$, 
\begin{eqnarray*}
S_f(\rho_\cN || \sigma_\cN) &=& \langle (\rho_\cN)^{1/2} , f(\Delta_{\sigma_\cN,\rho_\cN})  (\rho_\cN)^{1/2}\rangle \\
&\leq &  \left \langle  U (\rho_\cN)^{1/2} , f(\Delta_{\sigma,\rho})  U (\rho_\cN)^{1/2}\right \rangle\\
&= &  \left  \langle \rho^{1/2} ,  f(\Delta_{\sigma,\rho}) \rho^{1/2}\right \rangle
= S_f(\rho||\sigma)\ .\\
\end{eqnarray*}
This proves, following Petz, his monotonicity theorem for the quasi relative entropy $S_f$ for the operator convex function. 

Now consider the family of quasi relative entropies defined by functions $f_t(x)=(t+x)^{-1}$. Our immediate goal is to prove the inequality
\begin{equation}\label{opjen2}
S_{(t)}(\rho||\sigma)= \langle \rho^{1/2},   (t+\Delta_{\sigma,\rho})^{-1} \rho^{1/2}\rangle \geq  
 \langle \rho_\cN^{1/2},   (t+\Delta_{\sigma_\cN,\rho_\cN})^{-1}  \rho_\cN^{1/2}\rangle=S_{(t)}(\rho_\cN||\sigma_\cN) \ .
\end{equation}

\begin{lemma}\label{conlemA}
Let $U$ be a partial isometry embedding  a  Hilbert space $\cK$ into a Hilbert space $\cH$.
Let $B$ be an invertible  positive operator on $\cK$, $A$ be an invertible  positive operator on $\cH$, and suppose that 
$ U^*AU =   B$. Then for all $v\in \cK$,
\begin{equation}\label{resineq}
\langle v, U^* A^{-1} U v\rangle =  \langle v, B^{-1} v\rangle+ \langle w, A w\rangle\ ,
\end{equation}
where
\begin{equation}\label{resineq2}
w := UB^{-1}v - A^{-1}Uv\ .
\end{equation}
\end{lemma}

\begin{proof} We compute, using $U^*U = \one_\cK$,
\begin{eqnarray*}
\langle w, A w\rangle &=& \langle UB^{-1}v - A^{-1}Uv, A UB^{-1}v - Uv)\rangle \\
&=& \langle v,B^{-1}U^*AUB^{-1}v\rangle  -2\langle v, B^{-1}v\rangle  + \langle v, U^*A^{-1}U v\rangle\\
&=& -\langle v, B^{-1}v\rangle  + \langle v, U^*A^{-1}U v\rangle\ 
\end{eqnarray*} 
\end{proof}

\begin{proof}[Proof of Theorem~\ref{thm:main2}] 
We apply  Lemma~\ref{conlemA} with $A := (t + \Delta_{\sigma,\rho})$,  $B = (t +\Delta_{\sigma_\cN,\rho_\cN})$ and 
$v := (\rho_\cN)^{1/2}$, and with $U$ defined as above. The lemma's condition, $U^*AU = B$, follows from  \eqref{key} and the fact that $U^*U = \one_\cK$. Therefore, applying Lemma~\ref{conlemA} with  $U(\rho_\cN)^{1/2} = \rho^{1/2}$, 
\begin{eqnarray}
S_{(t)}(\rho||\sigma) - S_{(t)}(\rho_\cN||\sigma_\cN) &=& \langle \rho^{1/2},   (t+\Delta_{\sigma,\rho})^{-1} \rho^{1/2}\rangle -  
 \langle \rho_\cN^{1/2},   (t+\Delta_{\sigma_\cN,\rho_\cN})^{-1}   \rho_\cN^{1/2}\rangle\nonumber\\
 &=&  \langle w_t, (t + \Delta_{\sigma,\rho}) w_t\rangle 
 \geq  t\|w_t\|^2, \label{eq:diff}
 \end{eqnarray}
 where, recalling that  $U(\rho_\cN)^{1/2} = \rho^{1/2}$,
 \begin{equation}\label{resolv}
w_t := U (t +\Delta_{\sigma_\cN,\rho_\cN})^{-1} (\rho_\cN)^{1/2} -  (t + \Delta_{\sigma,\rho})^{-1}\rho^{1/2}\ .
\end{equation}

 Using the integral representation of the square root function,
 $$
 X^{1/2} = \frac{1}{\pi} \int_0^\infty t^{1/2} \left(\frac{1}{t} - \frac{1}{t+X}\right){\rm d}t,
 $$
 and $U(\cN \rho)^{1/2} = \rho^{1/2}$ once more, we conclude that
 $$
 U(\Delta_{\sigma_\cN,\rho_\cN})^{1/2} (\rho_\cN)^{1/2}  -  (\Delta_{\sigma,\rho})^{1/2}\rho^{1/2} =  \frac{1}{\pi}\int_0^\infty t^{1/2}w_t{\rm d}t\ .
 $$
 On the other hand,
 \begin{eqnarray*}
 U(\Delta_{\sigma_\cN,\rho_\cN})^{1/2} (\rho_\cN)^{1/2}  -  (\Delta_{\sigma,\rho})^{1/2}\rho^{1/2} &=& 
 U (\sigma_\cN)^{1/2} - \sigma^{1/2} \\
 &=& (\sigma_\cN)^{1/2}(\rho_\cN)^{-1/2} \rho^{1/2} - \sigma^{1/2}\ .
 \end{eqnarray*}
 Therefore, combining the last two equalities and taking the Hilbert space norm associated with $\cH$, for any $T>0$, 
\begin{eqnarray}\label{REM0}
\|(\sigma_\cN)^{1/2}(\rho_\cN)^{-1/2} \rho^{1/2} - \sigma^{1/2}\|_2 &=&
 \frac{1}{\pi}\left\Vert \int_0^\infty  t^{1/2} w_t{\rm d}t\right\Vert_2  \nonumber\\
&\leq&   \frac{1}{\pi} \int_0^T t^{1/2}  \|w_t\|_2{\rm d}t +  \frac{1}{\pi}\left\| \int_T^\infty t^{1/2} w_t {\rm d}t \right\|_2\ .
\end{eqnarray}
We estimate  these two terms separately. For the first term,  by the Cauchy-Schwarz inequality, 
\begin{eqnarray}\label{REM1}
\left(\int_0^T t^{1/2}  \|w_t\|_2{\rm d}t  \right)^2 &\leq& T \int_0^T t \|w_t\|_2^2{\rm d}t \nonumber\\
&\leq& T \int_0^\infty\left(  S_{(t)}(\rho||\sigma) - S_{(t)}(\rho_\cN||\sigma_\cN)  \right){\rm d}t \nonumber\\
&=& T (S(\rho||\sigma) - S(\rho_\cN||\sigma_\cN))\ .
\end{eqnarray}
For the second term in (\ref{REM0}), note that for any positive  operator $X$
$$
t^{1/2} \left(\frac{1}{t} - \frac{1}{t+X}\right)   \leq  t^{1/2} \left(\frac{1}{t} - \frac{1}{t+\|X\|}\right)\one  =  \frac{\|X\|}{t^{1/2}(\|X\|+ t)}\one,
$$
and hence
$$
\int_T^\infty t^{1/2} \left(\frac{1}{t} - \frac{1}{t+X}\right) {\rm d}t \leq \|X\|^{1/2}\left(\int_{T/\|X\|}^\infty \frac{1}{t^{1/2}(1+t)}  {\rm d}t\right) \one \leq \frac{2\|X\|}{T^{1/2}}\one\ .
$$
The spectra of $\sigma_\cN$ and $\rho_\cN$ lie in the convex hulls of the spectra of $\sigma$ and $\rho$ respectively. It follows that $\|\Delta_{\sigma_\cN,\rho_\cN}\|  \leq \|\Delta_{\sigma,\rho}\|$.  Therefore, recalling the definition of $w_t$ in (\ref{resolv}), we obtain
\begin{equation}\label{REM2}
\left\| \int_T^\infty t^{1/2} w_t {\rm d}t \right\|_2  \leq \frac{4 \|\Delta_{\sigma,\rho}\|}{T^{1/2}}\ .
\end{equation}
Combining \eqref{REM0}, \eqref{REM1} and \eqref{REM2} we obtain
$$\|(\sigma_\cN)^{1/2}(\rho_\cN)^{-1/2} \rho^{1/2} - \sigma^{1/2}\|_2 \leq  
 \frac{1}{\pi} T^{1/2} (S(\rho||\sigma) - S(\rho_\cN||\sigma_\cN))^{1/2}  \ +\  
  \frac{4 \|\Delta_{\sigma,\rho}\|}{\pi T^{1/2}} \ .$$
Optimizing in $T$,
$$
\|(\sigma_\cN)^{1/2}(\rho_\cN)^{-1/2} \rho^{1/2} - \sigma^{1/2}\|_2 \leq  \frac{4}{\pi} \|\Delta_{\sigma,\rho}\|^{1/2}
(S(\rho||\sigma) - S(\rho_\cN||\sigma_\cN))^{1/4}
$$
Rearranging terms
\begin{equation}\label{REM5}
S(\rho||\sigma) - S(\rho_\cN||\sigma_\cN)  \geq \left(\frac{\pi}{4}\right)^{4} \|\Delta_{\sigma,\rho}\|^{-2}
\|(\sigma_\cN)^{1/2}(\rho_\cN)^{-1/2} \rho^{1/2} - \sigma^{1/2}\|^4_2\ .
\end{equation}
\end{proof}

We now prove the lemma leading from \eqref{REM5a} to \eqref{REM5b}. 

\begin{lemma}
For any operators $X$ and $Y$ with $\Tr[X^*X] = \Tr[Y^*Y] =1$.  Then 
\begin{equation}\label{REM7}
\|X^*X - Y^*Y\|_1 \leq 2 \|X- Y\|_2\ .
\end{equation}
\end{lemma}

\begin{proof} Recall that for any operator $A$, $\|A\|_1 = \sup\{ |\Tr[ZA]\ :\ \|Z\| \leq 1\}$ where $\|\cdot\|$ denotes the operator norm. 
For any contraction $Z$, using cyclicity of the trace we have
\begin{eqnarray*}
|\Tr[Z(X^*X - Y^*Y)]| &\leq& |\Tr[Z(X^* - Y^*)X + ZY^*(X - Y)]|\\
&\leq& |\Tr[(X^* - Y^*)XZ| + |\Tr[ZY^*(X - Y)]\\
&\leq& (\Tr(X^*-Y^*)(X-Y)])^{1/2} (\Tr[X^*Z^*ZX])^{1/2}\\
&+& (\Tr(X^*-Y^*)(X-Y)])^{1/2} (\Tr[Y^*Z^*ZY])^{1/2}\\
&\leq& 2\|X- Y\|_2\ .
\end{eqnarray*}
\end{proof}

Applying this with $X = (\sigma_\cN)^{1/2}(\rho_\cN)^{-1/2} \rho^{1/2}$ and $Y= \sigma^{1/2}$,
we get
$$\|(\sigma_\cN)^{1/2}(\rho_\cN)^{-1/2} \rho^{1/2} - 
\sigma^{1/2}\|_2 \geq \frac12 \| \prr(\sigma_\cN) -\sigma\|_1\ .$$


\section{Structure of the solution set of the Petz equation}\label{Sec:Struc}

Define the CPTP map $\Phi: \cM \to \cM$ by
\begin{equation}\label{Phidef}
\Phi := \prr \circ \sE_\tau \ .
\end{equation}
The Petz equation \eqref{equal} can be written as $\Phi(\sigma) = \sigma$.   The adjoint of $\Phi$, is the completely positive unital map $\Phi^\dagger=\Psi: \cN\to \cN$ given by
\begin{equation}\label{Psidef}
\Psi := \iota_{\cN,\cM}\circ \sA_\rho \ 
\end{equation}
where $ \iota_{\cN,\cM}$ is the inclusion of $\cN$ in $\cM$.

The problem of determining all of the states fixed by $\Phi$ is closely related to the problem of determining all of the fixed points of $\Psi$ in $\cM$. This problem has been investigated in a general context  
by  Lindblad \cite{L99} in the  proof  of his  {\em General No-cloning Theorem}, drawing on earlier work by  Choi and Kadison. It was also investigated in this specific context by Accardi and Cecchini. 
For now, we need not assume that $\Psi$ is given by \eqref{Psidef}. {\em For now, all we require is that
$\Psi$ is a unital completely positive map from $\cM$ to $\cM$, and that its Hilbert-Schmidt dual 
$\Phi$ has a faithful invariant state $\rho$}.

Then, by an often used argument,  the map $\Psi$ is a contraction on $(\cM, \langle \cdot, \cdot\rangle_{GNS,\rho})$:
By the operator Schwarz inequality, for all $X\in \cM$, $\Psi(X)^*\Psi(X) \leq \Psi(X^*X)$. Then 
$$\|\Psi(X)\|_{GNS,\rho}^2 = \rho(\Psi(X)^*\Psi(X) ) \leq \rho(\Psi(X^*X)) = \Phi(\rho)(X^*X) = \|X\|_{GNS,\rho}^2\ .$$

Define
\begin{equation}\label{Cdef}
\mathcal{C}  = \{ X\in \cN \ : \ \Psi(X) = X\ \}\ , 
\end{equation}
which is evidently a subspace.
Let $\sE_{\mathcal{C} }$ be the orthogonal projection in $(\cM, \langle \cdot, \cdot\rangle_{GNS,\rho})$
onto ${\mathcal{C}}$.    
Then, arguing as in \cite{L99}, by  the von Neumann Mean Ergodic Theorem, 
$${\sE}_{\mathcal C}(X) = \lim_{N\to\infty}\frac{1}{N} \sum_{j=1}^N \Psi^{j}(X)\ ,$$
The following lemma may be found in \cite{L99}; we give the short proof for the reader's convenience. 

\begin{lemma}\label{Lindlem1} Let $\Phi$ be a CPTP map on $\cM$, and let $\Psi = \Phi^\dagger$. 
A density matrix $\tau\in \cM$ satisfies $\Phi(\tau) = \tau$ if and only if it satisfies 
${\sE}_{\mathcal C}^\dagger(\tau) = \tau$.
\end{lemma}

\begin{proof}
Let $\tau$ be any density matrix $\tau$ in $\cM$ such that $\Phi(\tau) = \tau$
For all $N$  and all $X\in \cM$, 
$$\Tr(\tau X) = \frac{1}{N} \sum_{j=1}^N(\Tr(\Phi^j(\tau)X) =  
\frac{1}{N}\sum_{j=1}^N \Tr(\tau \Psi^{j}(X))\ .$$
In the limit, we obtain 
$\Tr(\tau X) = \Tr(\tau{\sE}_{\mathcal C}(X)) = \Tr({\sE}_{\mathcal C}^\dagger(\tau) X)$.  Hence $\tau = {\sE}_{\mathcal C}^\dagger\tau$.

Now suppose that $\tau$ is any density matrix in $\cM$ satisfying  $\tau = {\sE}_{\mathcal C}^\dagger \tau$.
Since evidently, $\Psi \circ  {\sE}_{\mathcal C} =   {\sE}_{\mathcal C}\circ  \Psi = {\sE}_{\mathcal C}$,  
for all $X\in \cM$,
$$\Tr(\tau X) = \Tr({\sE}_{\mathcal C}^\dagger(\tau) X) = \Tr(\tau, ({\sE}_{\mathcal C}(X))) = 
\Tr(\tau, ({\sE}_{\mathcal C}(\Psi X))) = \Tr(\tau, \Psi (X))) = \Tr(\Phi(\tau) X)\ .$$
since $X$ is arbitrary, $\Phi(\tau) = \tau$. 
\end{proof}

Furthermore, by results of Choi \cite{C74} and Lindblad \cite{L99}, $\mathcal{C}$ is a  unital $*$-subalgebra of $\cN$.  Let ${\mathcal Z}$ denote the center of $\cC$, which is commutative von Neumann algebra. 
Because $\mathcal{Z}$ is commutative, it has a particularly simple structure: If $P$ and $Q$ are two orthogonal projections in $\mathcal{Z}$, then $PQ = QP$ is also an orthogonal projection in $\mathcal{Z}$.  Since $\mathcal{Z}$
is the closed linear span of the projections contained in it, one easily deduces the existence of a family 
$\{P_1, \dots,P_J\}$ of mutually orthogonal projections summing to the identity such that  $\mathcal{Z}$ is the span of these projections. 

Define $\cH^{(j)}$, $j= 1,\dots, J$, to be the range of $P_j$. Then the Hilbert space $\cH$ on which $\cM$ acts can be decomposed as 
${\displaystyle \cH = \bigoplus_{j=1}^J \cH^{(j)}}$.
(The notation is chosen to avoid confusion with tensor product decompositions such as, e.g., $\cH = \cH_1\otimes \cH_2$ for bipartite systems.)
Then each $\cH^{(j)}$ is invariant under $\mathcal{C}$, and the center of $\mathcal{C}$ restricted 
to each $\cH^{(j)}$ is trivial -- it is spanned by $P_j$, the identity on $\cH^{(j)}$.   Therefore, the 
restriction of $\mathcal{C}$  to each $\cH^{(j)}$ is a {\em factor} -- a $*$-subalgebra of 
$\cB(\cH)$ with a trivial center. By the well-known structure theorem for finite dimensional factors, 
$\cH^{(j)}$ can be factored as $\cH^{(j)} = \cH_{j,\ell}\otimes \cH_{j,r}$
and 
$$\mathcal{C}_j =   1_{\cH_{j,\ell}}\otimes \cB(\cH_{j,r}) \ .$$
Using this decomposition and structure theorem, Lindblad proves \cite[Section 4]{L99} the following, stated here in terms  of the notation set above:

\begin{lemma}\label{Lindlem2} Let $\Psi$ be a unital completely positive map on $\cM$, where $\cM$ acts on a finite dimensional Hilbert space $\cH$, and where $\Psi^\dagger$ leaves a faithful state $\rho$ invariant.
Let $\sE_{\cC}$ be the orthogonal projection onto $\cC$, the $C^*$ algebra of fixed points of $\Psi$, with respect to the GNS inner product induced by $\rho$.   Then there are
uniquely determined density matrices $\{\gamma_{1,\ell},\dots,\gamma_{J,\ell}\}$, where $\gamma_{j,r}$ acts on $\cH_{j,r}$,
such that for all $Y\in \cM$,
$$
\sE_{\cC}(Y) = \sum_{j=1}^J1_{\cH_{j,\ell}} \otimes \Tr_{\cH_{j,\ell}}[ ( \gamma_{j,\ell}\otimes 1_{\cH_{j,r}}) P_jYP_j] \ ,
 $$
 where $\Tr_{\cH_{j,\ell}}$ denotes the trace over $\cH_{j,\ell}$. 
\end{lemma} 
From this explicit description of $\sE_{\cC}$, one readily deduces that 
\begin{equation}\label{etaud}
\sE_\cC^\dagger(\tau) = \sum_{j=1}^J \gamma_{j,\ell}\otimes \Tr_{\cH_{j,\ell}}(P_j\tau P_j) \ 
\end{equation}
where $\gamma_{j,\ell}\otimes \Tr_{\cH_{j,\ell}}(P_j\tau P_j)$ is defined as operators on all of $\cH$
by setting it to zero on the orthogonal complement of $\cH^{(j)}$.   Hence $\tau = 
\sE_\cC^\dagger(\tau)$ if and only if $\tau$ is given by the right hand side of \eqref{etaud}.

Now return to the case at hand, in which $\Phi$ and $\Psi$ are given by 
\eqref{Phidef} and \eqref{Psidef} respectively.  

\begin{proof}[Proof of Theorem~\ref{thm:main}]
Since $\Phi\rho = \rho$, Lemma~\ref{Lindlem1}, Lemma~\ref{Lindlem2} and \eqref{etaud} yield
\begin{equation}\label{bonn1}
\rho =  \sum_{j=1}^J  \gamma_{j,\ell}\otimes \Tr_{\cH_{j,\ell}}(P_j\rho P_j)
\end{equation}
with the set $\{\gamma_{1,\ell},\dots,\gamma_{J,\ell}\}$ determined by $\cC$, the fixed point algebra of
$\Psi = \iota_{\cN,\cM}\cdot \sA_\rho$.

Next observe that $\cC$ is a von Neumann subalgebra of $\cN$, and is in 
fact that fixed point algebra of $\widetilde \Psi := \sA_\rho \circ \iota_{\cN,\cM}$ 
as well as of  $\Psi = \iota_{\cN,\cM} \circ  \sA_\rho$.  Since  
$\widetilde \Phi := \widetilde \Psi^\dagger = \sE_\tau\circ \sR_\rho$, we have 
$\widetilde \Phi \rho_\cN = \rho_\cN$ and 
$\widetilde \Phi \sigma_\cN = \sigma_\cN$.  Using Lemma~\ref{Lindlem1}, 
Lemma~\ref{Lindlem2} and \eqref{etaud} once more, we see that for some 
$\{\widetilde \gamma_{i,\ell},\dots,\widetilde \gamma_{J,\ell}\}$, where for each $j$,  
$\widetilde \gamma_{j,\ell}$ is a density matrix on $\cH_{j,\ell}$
\begin{equation}\label{bonn1}\rho_\cN = 
 \sum_{j=1}^J \widetilde \gamma_{j,\ell}\otimes \Tr_{\cH_{j,\ell}}(P_j\rho_\cN P_j)\ .
 \end{equation}
Now observe that we may factor
\begin{equation}\label{bonn3}
\rho =  \sum_{j=1}^J   (\gamma_{j,\ell}\otimes \one_{\cH_{j,r}})
(\one_{\cH_{j,\ell}} \otimes  (\Tr_{\cH_{j,\ell}}(P_j\rho P_j)))\ 
\end{equation}
For each $j$, $\one_{\cH_{j,\ell}} \otimes  \Tr_{\cH_{j,\ell}}(P_j\rho P_j)\in \cC\subset \cN$. 
Therefore,
\begin{equation}\label{bonn4}
\rho_\cN = \sE_\tau\rho = \sum_{j=1}^J   \sE_\tau (\gamma_{j,\ell}\otimes \one_{\cH_{j,r}})
(\one_{\cH_{j,\ell}} \otimes  (\Tr_{\cH_{j,\ell}}(P_j\rho P_j)))\ 
\end{equation}
We now claim that for each $j$,
\begin{equation}\label{bonn5}
\sE_\tau (\gamma_{j,\ell}\otimes \one_{\cH_{j,r}}) = \widetilde \gamma_{j,\ell}\otimes \one_{\cH_{j,r}}\ .
\end{equation}
To see this note that since ${\mathcal Z} \subset \cC \subset \cN$, $\cN' \subset \cC' \subset {\mathcal Z}' = \{P_1,\dots,P_J\}'$, every unitary in $\cN'$ commutes with each $P_j$
and thus has the block form $U = \sum_{j=1}^J P_j U P_j$. Moreover, again using the fact that
$\cN' \subset \cC'$, and that the commutator of $\one_{\cH_{j,\ell}}\otimes \cB(\cH_{j,r})$ is
$\cB(\cH_{j,\ell}) \otimes  \one_{\cH_{j,r}}$, we see that $U$ has the form
$$U = \sum_{j=1}^J P_j (U_{j,\ell}\otimes  \one_{\cH_{j,r}}) P_j, $$
where $U_{j,\ell}$ is unitary on $\cH_{j,\ell}$, though in general, only a subset of the 
block unitaries of this form belong to $\cN'$.   In any case, representing 
$\sE_\tau$ as an average over appropriate unitaries of this form \cite{D59,U73}, we obtain \eqref{bonn5}. 
Now combining \eqref{bonn3}, \eqref{bonn4} and \eqref{bonn5} yields
\begin{equation}\label{bonn6}
\rho_\cN = \sum_{j=1}^J  (\widetilde\gamma_{j,\ell}\otimes \one_{\cH_{j,r}})
(\one_{\cH_{j,\ell}} \otimes  (\Tr_{\cH_{j,\ell}}(P_j\rho P_j)))\ \ .
\end{equation}

Combing \eqref{bonn3} and \eqref{bonn6} it follows that 
${\displaystyle \rho^{-1}\rho_\cN = \sum_{j=1}^J  (\gamma_{j,\ell}^{-1}\widetilde\gamma_{j,\ell}\otimes \one_{\cH_{j,r}})}$.
The general element $A$ of $\cC$ has the from 
${\displaystyle A = \sum_{j=1}^J    1_{\cH_{j,\ell}}\otimes A_{j,r}}$ where each $A_{j,r} \in \cB(\cH_{j,r})$.
For such $A$,
\begin{equation}
\Delta_\rho^{-1}(\Delta_{\rho_\cN}(A)) = \rho_\cN^{-1}\rho A\rho^{-1}\rho_\cN
=  A\ ,
\end{equation}
and this verifies that $\Delta_\rho(A) = \Delta_{\rho_\cN}(A)$ for all $A\in \cC$, which we know must be valid by Theorem~\ref{takvar}.

The same analysis applies to $\sigma$ and $\sigma_\cN$ yielding
$
\sigma =  \sum_{j=1}^J   (\gamma_{j,\ell}\otimes \one_{\cH_{j,r}})
(\one_{\cH_{j,\ell}} \otimes  (\Tr_{\cH_{j,\ell}}(P_j\sigma P_j)))\ 
$
and 
$
\sigma_\cN = \sum_{j=1}^J  (\widetilde\gamma_{j,\ell}\otimes \one_{\cH_{j,r}})
(\one_{\cH_{j,\ell}} \otimes  (\Tr_{\cH_{j,\ell}}(P_j\sigma P_j)))\ ,
$ and then  $\Delta_\sigma(A) = \Delta_{\sigma_\cN}(A)$ for all $A\in \cC$.
\end{proof}

\begin{theorem}\label{largest} Let $\cC$ be defined by \eqref{Cdef} and $\Psi := \iota_{\cN,\cM}\circ \sA_\rho$, 
and let $\cB$ be any other von Neumann subalgebra of $\cN$ that is invariant under
$\Delta_\rho$. Then $\cB \subset\cC$.
\end{theorem}

\begin{proof}  Let $\sE_{\tau,\cB}$ denote the tracial conditional expectation onto $\cB$, and define
$\rho_\cB := \sE_{\tau,\cB}\rho$.  Let  $\sA_{\rho,\cB,\cN}$, $\sA_{\rho,\cN,\cM}$ and 
$\sA_{\rho,\cB,\cM}$ be the Accardi-Cecchini coarse graining operators from $\cN$ to $\cB$, $\cM$ to $\cN$ and $\cM$ to $\cB$. A simple computation shows that
\begin{equation}\label{cmposite}
\sA_{\rho,\cB,\cM}  = \sA_{\rho,\cB,\cN}\circ \sA_{\rho,\cN,\cM}\ .
\end{equation}

Let $\sP_\cB$ denote the orthogonal projection of $\cM$ onto $\cB$  with respect to the GNS inner product induced by $\rho$. Since $\Delta_\rho$ leaves $\cB$ invariant, by {\it (2)} of Theorem~\ref{takvar}  $\sA_{\rho,\cB,\cM} = \sP_\cB$.  We claim that $\sA_{\rho,\cB,\cN}$
is the restriction of $\sP_\rho$ to $\cN$. Indeed, by the defining relation \eqref{ac0},
for all $X\in \cM$ and all $Y\in \cB$,
\begin{equation}\label{ac0B}
 \langle X,Y\rangle_{KMS,\rho} = \langle \sA_{\rho,\cB,\cM}(X),Y\rangle_{KMS,\rho_\cB}\ .
\end{equation}
Tautologically, this holds for all $X\in \cN$ and all $Y\in \cB$, and so for all $X\in \cN$,
$\sA_{\rho,\cB,\cN}(X) = \sA_{\rho,\cB,\cM}(X)$. We therefore have that for all $B\in B$
\begin{equation}\label{cmposite}
B = \sP_\rho B   = \sP_\rho B\circ \sA_{\rho,\cN,\cM}(B)\ ,
\end{equation}
and this implies that $B= \sA_{\rho,\cN,\cM}(B)$, which, by definition, means that $B\in \cC$. 
\end{proof}


\section{Conditional expectations}\label{Sec:Exp}

Recall from  the introduction that if  $\rho$ is  a faithful state on $\cM$ and 
$\cN$ is a von Neumann subalgebra of $\cM$,  then
there exists a conditional expectation $\sE$    from $\cM$ to $\cN$ such that
for all $X\in \cM$, $\rho(X) = \rho(\sE(X))$ if and only if the orthogonal projection onto
$\cN$ in the GNS inner product induced by $\rho$ is a conditional expectation.

This raises the question: For which faithful states $\rho$ is   the orthogonal projection onto
$\cN$ in the GNS inner product induced by $\rho$ is actually a conditional expectation?

\if false
As explained in the introduction, after the work of Tomiyama \cite{To57}, 
we may say that a linear map
$\sE$ from $\cM$ to $\cN$ is a conditional expectation  if 
and only if it is norm one projection from $\cM$ onto $\cN$, and whenever 
this is the case, $\sE$ preserves positivity and satisfies \eqref{tomi} and \eqref{tomi2}. 
The following theorem is due to Umegaki.

\begin{theorem}[Umegaki]\label{Umtr} Let $\sE_\tau$ be the orthogonal projection onto 
$\cN$ in the Hilbert-Schmidt inner product, or, what is the same thing, with respect 
to the GNS inner product induced by $\tau$. Then $\sE_\tau$ is a conditional expectation 
in the sense of Umegaki. 
\end{theorem}

The proof we give of this relies on a formula due to Davis \cite{D59} and Uhlmann \cite{U73} that will 
be useful in the next section.  Recall that $\cM$ itself is a subalgebra (or possibly all of) $M_n(\C)$. 
The group ${\mathcal G}$ 
of unitary matrices $U$ such that $U\in \cN'$ is a compact Lie subgroup of the 
group of all $n\times n$ unitary matrices. Let $\mu$ denote normalized ($\mu({\mathcal G}) =1$) right-invariant Haar 
measure on ${\mathcal G}$. Define a linear operator $\sP: \cM\to \cM$ by
\begin{equation}\label{projdef}
\sP(X) = \int_{\mathcal G} U^*XU{\rm d}\mu(U)\ .
\end{equation}
By the right invariance of $\mu$, for any $U_0\in {\mathcal G}$, 
$$U_0^*\sP(X)U_0 = \int_{\mathcal G} U_0^*U^*XUU_0{\rm d}\mu(U) = \sP(X)\ .$$
Therefore, $\sP(X)$ commutes with every unitary in $\cN'$, and since $\cN'$ is 
generated by the unitaries it contains, $\sP(X)\in \cN''$.  Then by the von 
Neumann Double Commutant Theorem,
$\cN''  = \cN$. Hence $\sP(X) \in \cN$ for all $X\in \cM$. It is evident that $\sP(A) = A$ for all $A\in \cN$
and by convexity of the operator norm, $\|\sP(X)\| \leq \|X\|$ for all $X\in \cM$. 
One could invoke Tomiyama's Theorem to assert that $\sP$ is a conditional 
expectation from $\cM$ onto $\cN$, 
but in this case, matters are very simple: For all $U\in \cN'$,  all $A,B\in \cN$ and all $X\in \cM$,
$U^*AXBU = A(U^*XU)B$, and averaging over $U$ yields $\sP(AXB) = A\sP(X)B$, so that \eqref{tomi}
is satisfied.  Moreover, not only is \eqref{tomi2} satisfied, it is evident that $\sP$ is completely positive. 

Then by cyclicity of the trace, for all $X\in \cM$
\begin{equation}\label{cyc}
\tau(\sP(X)) = \int_{\mathcal G} \tau(U^*XU){\rm d}\mu(U) = \int_{\mathcal G} \tau(X){\rm d}\mu(U) =\tau(X) \ . 
\end{equation}
Thus, \eqref{tomi3} is satisfied, and as pointed out in the Introduction, this 
implies that $\sP$ is the orthogonal projection from $\cM$ onto $\cN$ in the 
GNS inner product induced by $\tau$, or, what is the same thing, in the Hilbert-Schmidt inner product. 

Having made these observations, it is a simple matter to prove Theorem~\ref{Umtr}, and more:

\begin{proof}[Proof of Theorem~\ref{Umtr}]  By that we have explained above, $\sE_\tau$, the orthogonal projection of $\cM$ onto $\cN$ in the GNS inner product induced by $\tau$ is given by the formula
\begin{equation}\label{projdef2}
\sE_\tau(X) = \int_{\mathcal G} U^*XU{\rm d}\mu(U)\ .
\end{equation}
and it is a conditional expectation in the sense of Umegaki; in particular \eqref{tomi}, 
\eqref{tomi2} and \eqref{tomi3} are all satisfied with $\sE_\tau$ in place of $\sE$ and $\tau$ in place of 
$\rho$. 
\end{proof}

The formula \eqref{projdef2}  will be used in the next section. This way of writing 
$\sE_\tau$ is due to Davis and Uhlmann, and  one can even replace ${\mathcal G}$ 
with a finite group of unitaries in $\cN'$, and $\mu$ with the normalized counting 
measure on this finite group \cite{D59,U73}. 

In \eqref{cyc} we made use of cyclicity of the trace.  When the state $\rho$ is not cyclic; i.e., when $\rho$ does not belong to the center of $\cM$,  the orthogonal projection of $\cM$ onto $\cN$ will be a conditional expectation in the sense of Umegaki only when $\cN$ is a very special sort of von Neumann subalgebra of $\cM$. 

\fi

\begin{theorem}\label{takvar} Let $\cM$ be a finite dimensional von Neumann algebra, and let $\cN$ be a von Neumann subalgebra of $\cM$. Let $\rho$ be  a faithful state on $\cM$, and let $\Delta_\rho$ be the {\em modular operator} on $\cM$ defined by $\Delta_\rho(X) = \rho X \rho^{-1}$. Let $\sP_\rho$ be the
orthogonal projection from $\cM$ onto $\cN$ in the GNS inner product induced by $\rho$.  Then:

\smallskip

\noindent{\it (1)}  $\sP_\rho$ is real; i.e., it preserves self-adjointness, if and only if  $\cN$ is invariant under $\Delta_\rho$.

\smallskip

\noindent{\it (2)} $\cN$ is invariant under $\Delta_\rho$ if and only if for all $A \in \cN$, 
\begin{equation}\label{sameaut}
\Delta_\rho(A) = \Delta_{\rho_\cN}(A)\ ,
\end{equation}
in which case $\Delta^t_\rho(A) = \Delta^t_{\rho_\cN}(A)$ for all $t\in \R$.  Furthermore,
\eqref{sameaut} is valid for all $A\in \cN$ if and only if  $\sA_\rho(A) = A$ for all $A\in \cN$.
\end{theorem}

\begin{remark} Part {\it (2)} of Theorem~\ref{takvar} is due to Accardi and Cecchini \cite[Theorem 5.1]{AC82}. In our finite dimensional context,  we give a very simple proof; most of the proof below is devoted to {\it (1)}.
\end{remark}

\begin{proof}[Proof of Theorem~\ref{takvar}]  Suppose that $\sP_\rho$ is real. Then for all $X\in {\rm Null}(\sP_\rho)$,
$0 = (\sP_\rho(X))^* = \sP_\rho(X^*)$, so that ${\rm Null}(\sP_\rho)$ is a self adjoint subspace of
$\cM$. Let $m$ denote the dimension of ${\rm Null}(\sP_\rho)$. Then, applying the Gram-Schmidt Algorithm, one can produce an orthonormal basis
$\{H_1,\dots,H_m\}$ of ${\rm Null}(\sP_\rho)$ consisting of self-adjoint elements of $\cM$. 

The map $X\mapsto X\rho^{1/2}$ is unitary from $(\cM, \langle \cdot, \cdot\rangle_{GNS,\rho})$
to $(\cM, \langle \cdot, \cdot\rangle_{HS,\rho})$.  Therefore for all $A\in \cN$, and each $j=1,\dots, m$, 
$\langle A\rho^{1/2}, H_j\rho^{1/2}\rangle_{HS}  = 0$. Then since the map $X\mapsto X^*$ is an (antilinear) isometry on $(\cM, \langle \cdot, \cdot\rangle_{HS,\rho})$,
$$0 = \langle (H_j\rho^{1/2})^*, (A\rho^{1/2})^*,\rangle_{HS}  = \Tr[   H_j\rho A^*] = 
\Tr[ H_j \Delta_\rho(A^*)\rho] =
  \langle H_j, \Delta_\rho(A^*),\rangle_{GNS,\rho} \ .$$
 Therefore, $\Delta_\rho(A^*)$ is orthogonal to $ {\rm Null}(\sP_\rho)$ in  
 $(\cM, \langle \cdot, \cdot\rangle_{GNS,\rho})$, and hence $\Delta_\rho(A^*)\in \cN$.  Since $A$ is arbitrary in $\cN$, it follows that $\cN$ is invariant under $\Delta_\rho$. 
 
For the converse, suppose that $\cN$ is invariant under $\Delta_\rho$. Then $\cN$ is invariant under
$\Delta_\rho^s$ for all $s\in \R$, and in particular, $\cN$ is invariant under $\Delta_\rho^{/12}$.
Then $\rho^{1/2}\cN = \cN\rho^{1/2}$ as subspaces  of $\cM$; let $\cK$ denote this 
subspace of $\cM$, which is evidently self-adjoint.   Let $H = H^*\in \cM$. Then there 
are uniquely determined $A,B\in \cN$ such that $H\rho^{1/2} - A\rho^{1/2}$ and 
$\rho^{1/2}H - \rho^{1/2}B$ are both orthogonal to $\cK$ in the Hilbert-Schmidt inner product. Thus,
$$H\rho^{1/2}  = (H\rho^{1/2} - A\rho^{1/2}) +  A\rho^{1/2}\qquad{\rm and} \qquad
\rho^{1/2}H = (\rho^{1/2}H - \rho^{1/2}B) + \rho^{1/2}B\ $$
are the orthogonal decompositions of $H\rho^{1/2}$ and $\rho^{1/2}H$ with respect to $\cK$. 
Again since $X\mapsto X\rho^{1/2}$ is unitary from $(\cM, \langle \cdot, \cdot\rangle_{GNS,\rho})$
to $(\cM, \langle \cdot, \cdot\rangle_{HS,\rho})$, $\sP_\rho(H) = A$. We must show that $A = A^*$.

Since $X\mapsto X^*$ is an isometry for the Hilbert-Schmidt inner product, and since 
$\cK$ is self adjoint,
$$\rho^{1/2}H = (\rho^{1/2}H - \rho^{1/2}A^*) + \rho^{1/2}A^*$$
is again an orthogonal decomposition of $\rho^{1/2}H$ with respect to $\cK$, 
and by uniqueness, $B = A^*$. Thus,
$$\rho^{1/2}H = (\rho^{1/2}H - \rho^{1/2}A^*) + \rho^{1/2}A^* \ .$$
Now apply $\Delta_\rho^{-1/2}$ to both sides to obtain
$H\rho^{1/2}  = (H\rho^{1/2} - A^*\rho^{1/2}) +  A^*\rho^{1/2}$. 
We claim that $H\rho^{1/2} - A^*\rho^{1/2}$ is orthogonal to $\cK$. Once this is shown, it will follow that 
$H\rho^{1/2}  = (H\rho^{1/2} - A^*\rho^{1/2}) +  A^*\rho^{1/2}$ is the orthogonal decomposition of 
$H\rho^{1/2}$ with respect to $\cK$. Again by uniqueness of the orthogonal decomposition, it will follow
that $A =A^*$.

Hence it remains to show that $H\rho^{1/2} - A^*\rho^{1/2}$ is orthogonal to 
$\cK$ in the Hilbert-Schmidt inner product. The general element of $\cK$ 
can be written as $\rho^{1/2}Z$ for $Z\ in \cN$. 
Then
$$\langle \rho^{1/2}Z, H\rho^{1/2} - A^*\rho^{1/2}\rangle_{HS} = 
\Tr[Z^*(\rho^{1/2}H - \rho^{1/2}A^*)\rho^{1/2} = \langle Z\rho^{1/2}, 
(\rho^{1/2}H - \rho^{1/2}A^*)\rangle_{HS}\ .$$ But we have seen above
that $\rho^{1/2}H - \rho^{1/2}A^*$ is orthogonal to $\cK$, and $Z\rho^{1/2}\in \cK$.
This proves {\it (1)}. 

To prove {\it (2)}, note first of all that when \eqref{sameaut}) is valid for all $A\in \cN$, then $\Delta_{\rho}$ preserves $\cN$ since the right side evidently belongs to $\cN$. 

Now suppose the $\Delta_\rho$ preserves $\cN$. Let $A,B\in \cN$. Then $A^*\Delta_{\rho_\cN}(\delta_\rho(B)))
 \in \cN$, and then by the definition of $\sE_\tau$ and cyclicity of the trace, 
$$\Tr[\rho(A^*\rho_\cN
\rho^{-1}B\rho \rho_\cN^{-1})] = \Tr [A^*\rho_\cN
\rho^{-1}B\rho] = \Tr [\rho_\cN(\rho^{-1}B\rho A^*)] \ .$$
In the same way, using the fact that $(\rho^{-1}B\rho A)\in \cN$ and cyclicity of the trace,
$$\Tr[\rho_\cN(\rho^{-1}B\rho A^*)] = \Tr[B\rho A^*] = \Tr [\rho A^*B]\ .
$$
Altogether, $\langle A, \Delta_{\rho_N}(\Delta_\rho^{-1}(B))\rangle_{GNS,\rho} = 
\langle A,B\rangle_{GNS,\rho}$. Since $\Delta_{\rho_N}(\Delta_\rho^{-1}(B))\in 
\cN$, and $A$ is arbitrary in $\cN$, $\Delta_{\rho_N}(\Delta_\rho^{-1}(B)) = B$, and hence 
$\Delta_\rho^{-1}(B) = \Delta_{\rho_N}^{-1}(B)$.  Then $\Delta_\rho^{-n}(B) = \Delta_{\rho_N}^{-n}(B)$
for all $n\in \N$, and then it follows that $\Delta_\rho^{t}(B) = \Delta_{\rho_N}^{t}(B)$ for all $t\in \R$. 

Finally, we show that \eqref{sameaut} is valid for all $A\in \cN$, then $\sA_\rho(A) = A$ for all $A\in \cN$:
$$\sE_{\tau}(\rho^{1/2}A\rho^{1/2}) = \sE_{\tau}(\Delta^{1/2}_{\rho}(A) \rho) =
\Delta^{1/2}_{\rho_\cN}(A)\sE_{\tau}(\rho) = \rho_{cN}{A}^{1/2} A\rho_\cN^{1/2}\ .$$

By {\it (2)} of Theorem~\ref{takvar}, for all $A\in \cN$, $\Delta^{1/2}_{\rho_\cN}(A) = 
\Delta^{1/2}_\rho(A)$,
and therefore 
$$\sE_{\tau,\cN}(\rho^{1/2}A\rho^{1/2}) = \sE_{\tau}(\Delta^{1/2}_{\rho_\cN}(A) \rho) =
\Delta^{1/2}_\rho(A)\sE_{\tau}(\rho) = \rho_\cN^{1/2} A\rho_\cN^{1/2}\ .$$
That is,
$A = \rho_{\cN}^{-1/2}\sE_{\tau}(\rho^{1/2}A\rho^{1/2}) \rho_\cN^{-1/2} = \sA_\rho(A)$. 
On the other hand, when $A  = \sA_\rho(A)$ for all $A\in \cN$,  $\sA_\rho$ is a norm one projection onto $\cN$, and by Tomiyama's Theorem \cite{To57}, it is a conditional expectation, and it satisfies 
$\rho(\sA_\rho(X)) X$ for all $X\in \cM$. Therefore, it must coincide with $\sP_\rho$,  the orthogonal projection form
$\cM$ onto $\cN$ in the GNS inner product induced by $\rho$.  Hence $\sP_\rho$ is a conditional expectation. By what we proved earlier, this means that $\cN$ is invariant under $\Delta_\rho$,
and then that \eqref{sameaut} is valid for all $A\in \cN$.
\end{proof}

\begin{theorem}\label{equivalences}  Let $\sP_\rho$ denote the orthogonal projection of $\cM$ onto 
$\cN$ in the GNS inner product induced by $\rho$. Then 

\smallskip
\noindent{\it (1)} $\sP_\rho$ is a conditional expectation if and only if $\cN$ is invariant under $\Delta_\rho$. 

\smallskip
\noindent{\it (2)} $\sP_\rho$ is a conditional expectation if and only if $\sP_\rho$ is real.

\end{theorem}

\begin{proof}   Theorem~\ref{takvar}  says that when 
$\Delta_\rho$ does not leave $\cN$ invariant, $\sP_\rho$ is not even 
real, and hence is not a conditional expectation.  On the other hand, when  
$\Delta_\rho$ leaves $\cN$ invariant, 
a theorem of Takesaki says that there exists a projection $\sE$  with unit 
norm from $\cM$ onto $\cN$ that satisfies \eqref{tomi3}.
By Tomiyama's Theorem and remarks we have made in the 
introduction,  this means that $\sE = \sP_\rho$, and  that $\sP_\rho$ is a conditional expectation in the sense of Umegaki. This proves {\it (1)}.

It is evident that if $\sP_\rho$ is a conditional expectation, this $\sP_\rho$ is real. On the other hand, if
$\sP_\rho$ is real, then by  Theorem~\ref{takvar}, $\cN$ is invariant under $\Delta_\rho$, and now {\it (2)} follows from {\it (1)}. 
\end{proof}


\section{Strong Subadditivity}\label{Sec:SSA}

We recall the proof of equivalence of the strong-subadditivity relation and the monotonicity of relative entropy under partial traces, according to \cite{LR73}, in which it is shown that strong sub-additivity relation can be written in the following form: for $\cH=\cH_1\otimes\cH_2\otimes\cH_3$ and $\rho_{123}\in \cB(\cH)$,
\begin{equation}\label{eq:SSA}
S(\rho_{12}|| \rho_1\otimes\rho_2)\leq S(\rho_{123}||\rho_1\otimes\rho_{23}), 
\end{equation}
where $\rho_{12} = \Tr_{\cH_3}\rho_{123}$ etc. (See \cite{LR73}). With $\cN := \cB(\cH_1\otimes\cH_2)$, 
$$S(\rho_{12}|| \rho_1\otimes\rho_2) = S((\rho_{123})_\cN|| (\rho_1\otimes\rho_{23})_\cN) \ .$$

The DPI inequality yields  
\begin{equation}\label{eq:mono-N-vN2}
S(\rho_\cN||\sigma_\cN)\leq S(\rho||\sigma),
\end{equation}
for $\rho,\sigma\in\cB(\cH)$.

\begin{lemma}(Lieb, Ruskai \cite{LR73})
Let $\cH=\cH_1\otimes\cH_2\otimes\cH_3$. The monotonicity of the relative entropy under partial traces holds for all states $\rho,\sigma\in\cB(\cH)$ if and only if the strong sub-additivity inequality 
\begin{equation}\label{SSAineq}
S(\rho_{12})+S(\rho_{23})  - S(\rho_{123})- S(\rho_2) \geq 0.
\end{equation}
holds for all states $\rho_{123}\in\cH_1\otimes\cH_2\otimes\cH_3$. 
\end{lemma}
\begin{proof}
(MONO$\Rightarrow$ SSA) 
From here it is clear that taking the CPTP to be a partial trace over the third space and $\rho=\rho_{123}$, $\sigma=\rho_1\otimes\rho_{23}$ in (\ref{eq:mono-N-vN2}) leads to (\ref{eq:SSA}). 

(SSA$\Rightarrow$ MONO) Let us take the space $\cH_3$ to be 2-dimensional and the state $\rho_{123}$ in the following form
\begin{equation}\label{eq:state-linear}
\rho_{123}=\lambda\rho'_{12}\otimes E_3+(1-\lambda)\rho''_{12}\otimes F_3,
\end{equation} 
where $E_3$ and $F_3$ are orthogonal one-dimensional projections on $\cH_3$ and $\lambda\in[0,1]$. Then the SSA relation (\ref{eq:SSA}) for this $\rho_{123}$ is equivalent to the concavity property of the conditional entropy {$S(\rho_{12}\|\rho_1)$}, i.e. for any $\lambda\in[0,1]$ and any $\rho'_{12}$ and $\rho''_{12}$ above we have
\begin{equation}\label{eq:cond-eq}
\lambda (S(\rho'_{12})-S(\rho'_2))+(1-\lambda)(S(\rho''_{12})-S(\rho''_2))\leq S(\rho_{12})-S(\rho_2).
\end{equation}

Recall that a function $f:[0,1]\rightarrow \mathbb{R}$ is called operator convex if for all matrices $A,B$ with eigenvalues in $[0,1]$ and $0<\lambda<1$ the following holds
\begin{equation}\label{eq:convex-f}
f(\lambda A+(1-\lambda)B)\leq \lambda f(A)+(1-\lambda)f(B).
\end{equation}
Note that if an operator concave function $f$ is  homogeneous (i.e. $f(t A)=t f(A)$ for all $t>0$), then for positive matrices $A$ and $B$
\begin{equation}\label{eq:derivative}
\frac{d}{dt}\bigg|_{t=0}f(A+t B):=\lim_{t\rightarrow 0}\frac{1}{t}\{f(A+t B)-f(A)\}\geq f(B)
\end{equation}
the above limit exits. To see this, use the homogeneity first, and then the concavity of $f$ in the following way
\begin{align}
f(A+tB)&=(1+t)f\left(\frac{1}{1+t}A+\frac{t}{1+t}B\right)\\
&\geq (1+t)\left(\frac{1}{1+t}f(A)+\frac{t}{1+t}f(B)\right)\\
&=f(A)+tf(B).
\end{align}

Take a conditional entropy as this function $f$:
$$ f(\gamma_{12}):=S(\gamma_{12})-S(\gamma_2).$$
Then the derivative is 
$$\frac{d}{dt}f(\gamma_{12}+t \omega_{12})=-\Tr\omega_{12}\ln(\gamma_{12}+t \omega_{12})+\Tr\omega_{2}\ln(\gamma_{2}+t \omega_{2}).$$
Since the conditional entropy is concave and homogeneous, applying inequality (\ref{eq:derivative}) leads to
the monotonicity of the relative entropy under partial traces (\ref{eq:mono-N-vN2}). 
\end{proof}

The stability bound proved here has obvious consequences for the SSA inequality, and can be used to give a quantitative version of the result \cite{H03} of Hayden, Josza, Petz and Winter.  For another improvement to the SSA inequality, namely
\begin{equation}\label{SSAineq2}
S(\rho_{12})+S(\rho_{23})  - S(\rho_{123})- S(\rho_2) \geq 2\max\{ S(\rho_1) - S(\rho_{13}), S(\rho_3) - S(\rho_{13}) \}\ ,
\end{equation}
see \cite{CL12}.

\vspace{0.3in}
\textbf{Acknowledgments.} The authors are grateful to Mark Wilde and Lin Zhang for comments and questions that have led us to add some reformulations in this version. EAC was partially supported by NSF grant DMS  1501007. AV is grateful to EAC for hosting her visits to Rutgers University, during which this work was partially completed. AV is partially supported by NSF grant DMS 1812734.

\end{document}